\documentclass[a4paper,11pt,english]{amsart}
\usepackage{amsthm,amsmath,amssymb}
\usepackage[T1]{fontenc}
\usepackage[utf8]{inputenx}
\usepackage{graphicx}
\usepackage{geometry}
\usepackage{tikz}
\usepackage{hyperref}
\usepackage{leftidx}
\usepackage{color}
\usepackage{pifont}
\usepackage{babel}

\newtheorem{theorem}{Theorem}

\newtheorem{lemma}[theorem]{Lemma}
\newtheorem{proposition}[theorem]{Proposition}
\newtheorem*{definition}{Definition}

\newtheorem*{remark}{Remark}

\newtheorem{corollary}[theorem]{\bf Corollary}
 
\usepackage{thmtools}
\usepackage{cleveref}

\newcommand{\E}{\mathbb{E}}
\newcommand{\N}{\mathbb{N}}
\newcommand{\Q}{\mathbb{Q}}

\renewcommand{\P}{\mathbb{P}}
\newcommand{\R}{\mathbb{R}}
\newcommand{\Z}{\mathbb{Z}}
\renewcommand{\L}{\mathbb{L}}
\newcommand{\Pbarre}{\overline{\mathbb{P}}}
\newcommand{\Ebarre}{\overline{\mathbb{E}}}
\newcommand{\1}{1\hspace{-1.3mm}1}

\newcommand{\Card}{\operatorname{Card}}
 \def \d {{\delta}}
 \def \l {{\lambda}}
 
 \def \g {{\gamma}}
 \def \w {{\omega}}
 
 \def \supp {{\text{ supp }}}

\title{The contact process with aging}
\author{Aurelia Deshayes}
\address{Institut \'Elie Cartan Nancy (math{\'e}matiques)\\Universit{\'e} de Lorraine\\Campus Scientifique, BP 239 \\54506 Vandoeuvre-l{\`e}s-Nancy  Cedex France\\}
\email{Aurelia.Deshayes@univ-lorraine.fr}

\subjclass[2000]{60K35, 82B43.}

\keywords{contact process, asymptotic shape theorem, renormalization, block construction, interacting particles system}

\begin{document}
\begin{abstract}
In this article, we introduce a contact process with aging: in this generalization of the classical contact process, each particle has an integer age that influences its ability to give birth. We prove here a shape theorem for this process conditioned to survive. In order to establish some key exponential decays, we adapt the Bezuidenhout and Grimmett construction~\cite{bezgrim} to build a coupling between our process and a supercritical oriented percolation. Our results also apply to the two-stage contact process introduced by Krone~\cite{krone}.\end{abstract}
\maketitle

\section{Introduction}

The \textit{classical contact process} was introduced by Harris~\cite{harris74} as an interacting particles system modelling the spread of a population on the sites of $\Z^{d}$. It is more specifically defined as a continuous-time Markov process $\{\xi_t\in\{0,1\}^{\Z^{d}},t\geq0\}$ which splits up $\Z^{d}$ into 'living' sites ($\{x\in\Z^d:\xi_t(x)=1\}$) and 'dead' sites ($\{x\in\Z^d:\xi_t(x)=0\}$) and which evolves along the following two rules:
\begin{enumerate}
 \item if a site is alive then it dies at rate $1$;
 \item\label{hic} if a site is dead then  it turns alive at rate $\lambda$ times the number of its living neighbors.
\end{enumerate}
Note that the above point~\ref{hic} can also be read as follows: a living site gives birth to its neighbors at rate $\lambda$, which clearly justifies the name of contact process.

Krone~\cite{krone} later introduced a variant of this model, the so-called \textit{two-stage contact process}, for which there exists an intermediate 'juvenile' type that is forced to mature before it can produce any offspring. The state of the process at time $t$ is given by an element $\xi_t\in\{0,1,2\}^{\Z^d}$. A site in state $0$ is still interpreted as dead but living sites are now classified in two categories: a site in state $1$ is considered as 'young', while a site in state 2 is seen as 'adult'. The model then evolves along quite different rules: only adults can give birth (at rate $\lambda$), each new offspring is considered as young, and transition from young to adult occurs at rate $\gamma$.
The aim of this paper is to extend both above-mentioned models by considering a contact process \textit{with aging}, that is a model where each living particle is endowed with an age in $\N^*=\{1,2,\ldots\}$ (we denote by $\N=\{0,1,2,\ldots\}$). The state of this process at time $t$ is therefore described by an element $\xi_t\in\N^{\Z^d}$, and its dynamics are defined in the same spirit as above: a dead site turns alive (with age one) proportionally to the ages of its so-called  'fertile' neighbors, 
while a living site ages at rate $\gamma$ and dies at rate $1$. Full details regarding this evolution procedure, as well as a convenient graphical construction of the process, will be provided in Section~\ref{model}.

In~\cite{krone}, Krone shows that the two-stage contact process is additive and monotone with respect to its parameters. He also exhibits some bounds on the supercritical region and concludes the paper with a few open questions related to the problem. Recently, Foxall has solved some of these problems in~\cite{foxall}: he has shown, for example, that finite survival and infinite survival are equivalent and he also proved that the complete convergence holds. With these considerations in mind, and going back to our model, our first task in Section~\ref{Ssurvival} will be to identify suitable conditions on the parameters for the process to survive.

A crucial issue regarding the contact process (or its extensions to a random/randomly-evolving environment) is to understand the behaviour of the process in the different regions. For instance, in~\cite{steifwar}, Steif and Warfheimer show that the critical contact process in a randomly-evolving environment dies out, while in~\cite{GMPCRE}, Garet and Marchand establish an asymptotic shape theorem for the contact process in random environment. We will elaborate on the process in the survival case through the exhibition of a shape
theorem: there exists a norm $\mu$ on $\R^{d}$ such that the set $H_t$ of points alive before time $t$ satisfies for every $\epsilon>0$, almost surely, for every large $t$,
\[(1-\epsilon)B_{\mu}\subset\frac{\tilde H_t}t\subset(1+\epsilon)B_{\mu},\]
where $\tilde{H_t}=H_t+[0,1]^d$, $B_\mu$ is the unit ball for $\mu$.

Let us briefly sketch out the main steps towards this shape theorem. Actually, in a similar way as
in the case of the classical contact process, the proof divides into two distinct parts:
\begin{itemize}
\item We first prove that, as soon as the process survives, its growth is at least linear. To do so, we will resort to a similar construction as the one given by Bezuidenhout and Grimmett in~\cite{bezgrim}, and show that a supercritical process conditioned to survive stochastically dominates a two-dimensional supercritical oriented percolation (Section~\ref{Sbezgrim}). The required growth controls indeed follow from this construction (Section~\ref{Scontrols}). As an additional by-product of the percolation construction we obtain the fact that Krone's process dies out on the critical region. 
 \item The second step will consist in exhibiting a shape result knowing that the growth is of linear order. This is analogue to the result of Durrett and Griffeath for the classical contact process~\cite{durrettgriffeath}, and to the result of Garet and Marchand~\cite{GMPCRE} in the random environment case.
\end{itemize}

Let us return to on the second step. Just as in the classical case, the contact process with aging is a non-permanent model for which extinction is possible, and we therefore look for a shape theorem for the process conditioned to survive. The first step is to prove the convergence of the hitting times in one fixed direction $x\in\Z^d$, that is the convergence, as $n$ tends to infinity, of the ratio $\frac{t(nx)}{n}$ where $t(y)$ stands for the first time when the process hits $y$. In 1965, Hammersley and Welsh~\cite{hammwelsh} introduced subadditive techniques to show such a convergence in the case of first passage percolation. Since then, proofs of this type of convergence use the Kingman subadditive ergodic theorem~\cite{king73} and its extensions, which appeal to subadditivity, stationarity and integrability assumptions on the system. 

In non-permanent models such as the one described by our process, the extinction is possible and the hitting times can be infinite so that the standard integrability conditions are not satisfied. On the other hand, if we condition the model to survive, then stationarity and subadditivity properties can be lost. To overcome such lacks, we will lean on an intermediate quantity, introduced in~\cite{GMPCRE} and denoted by $\sigma(x)$, which morally represents a time when the site $x$ is alive and has infinitely many descendants (see Definition~\ref{defsigma}). It turns out that this function $\sigma$ satisfies adequate stationarity properties as well as the almost-subaddivity conditions involved in Kesten and Hammersley’s theorem ~\cite{kestentrick,hamm74}, a well-known extension of Kingman’s seminal result. The last part of
the proof will then be devoted to the control of the difference between $\sigma$ and $t$, and it will allow us to turn the shape theorem for $\sigma$ into the expected shape theorem for $t$.

\section{Model and results}\label{model}
\subsection{Evolution and initial set} 
\subsubsection{Birth parameters} In the contact process with aging, the infection parameter $\lambda$ is replaced with a sequence $\Lambda=(\lambda_i)_{i\in\N}$ of birth parameters. We assume that:
\begin{enumerate}
 \item \label{initial} $\forall i,\l_i\in\R^+$ and $\lambda_0=0$,
 \item \label{croissante} $(\lambda_i)_i$ is non decreasing,
 \item \label{fini} $\lim_{i\rightarrow\infty}\lambda_i=\lambda_\infty<\infty$.
\end{enumerate}
The quantity $\lambda_i$ is the birth parameter for a particle of age $i$. The condition~\ref{fini}  means that the rate of birth is bounded for every age; the condition~\ref{croissante} means that the rate of birth is non decreasing with age. These two conditions are reasonable for the model and are useful for coupling in the sequel. The condition~\ref{initial} account in particular for the fact that a dead site can not give birth to its neighbors.

\subsubsection{Markov process} The contact process with aging (CPA) is a continuous-time homogeneous Markov process $(\xi_t)_{t\geq0}$ taking values in the set of maps from $\Z^d$ to $\N$ (we can also see $(\xi_t)$ as a partition of points of $\Z^d$ according to their age). If $\xi_t(z)=0$ we say that $z$ is dead, while if $\xi_t(z)=i\ge 1$ we say that $z$ is alive with age $i$.

\subsubsection{Evolution} The evolution of the process is as follows:
\begin{itemize}
\item a living site dies at rate $1$ independently of its age,
\item a dead site $z$ turns alive at rate $\sum_{z',\|z'-z\|_1=1}\lambda_{\xi_t(z')}$ (with $\lambda_0=0$),
\item each new offspring has age one,
\item the transition from age $n$ to age $n+1$ occurs at rate $\gamma>0$, independently of its age. 
\end{itemize}
These evolutions are all independent from each other. In the following, we denote by $\mathcal{D}$ the set of cadlag functions from $\R^+$ to $\N^{\Z^d}$: it is the set of trajectories for Markov processes with state space given by the maps from $\Z^d$ to $\N$.

\subsubsection{Initial set} Let $f:\Z^d\to \N$ be a function with finite support. We start the process from the initial configuration $\xi_0=f$. The set $\{x\in\Z^d:f(x)\neq0\}$ is the set of living points and each such $x$ has age $f(x)$. Notations in special cases:
\begin{itemize}
 \item If it is not specified (as in $\xi_t$), then the initial set is the point $0$ with age  $1$. 
 \item If the initial set is $A$, where all particles have age 1, then we denote the corresponding process by $\xi_t^A$.
 \item If the initial set is a single point $x\in\Z^d$ with age $k\in\N^*$, then we denote the corresponding process by $\xi_t^{k\d_x}$. If $k=1$, then we denote the corresponding process by $\xi_t^x$.
\end{itemize}

Our model is a generalization of Krone's model~\cite{krone}. Indeed, if we set $\l_1=0$ and for every $i\geq 1$, $\l_i=\l$ we get Krone's infection pattern of parameters $(\l,\g)$ (with $\d=0$).

Just as in the contact process situation, we use Harris' procedure to construct the CPA on an appropriate probability space.

\subsection{Construction of the probability space} 

We consider $\mathcal{B}(\R_+)$ the Borel $\sigma$-algebra on $\R^+$ and $\E^d$ the edges of $\Z^d$. Let $M$ be the set of locally finite counting measures on $\R_+$ ($m=\sum_{i=0}^\infty\d_{t_i},(t_i)_i\in\R^+$ with $m(K)<\infty$ for every compact set $K$). We endow $M$ with the $\sigma$-algebra $\mathcal{M}$ generated by the maps ($m\rightarrow m(B), B \in\mathcal{B}(\R^+)$). Let
\begin{align*}
  \Omega&=\left(M\times[0,1]^\N\right)^{\E^d}\times \left(M^2\right)^{\Z^d}\\
  \text{ and } \mathcal{F}&=\left(\mathcal{M}\otimes\mathcal{B}\left([0,1]\right)^{\otimes \N}\right)^{\otimes\E^d}\otimes\left(\mathcal{M}\otimes\mathcal{M}\right)^{\otimes \Z^d}.
\end{align*}
On this space we consider the family of probability measures defined as follows
\begin{align*}
\P_{\Lambda,\g}=\left(\mathcal{P}_{\lambda_\infty}\otimes{\mathcal{U}_{[0,1]}}^{\otimes\N}\right)^{\otimes\E^d}\otimes\left(\mathcal{P}_1\otimes\mathcal{P}_{\gamma}\right)^{\otimes\Z^d}
\end{align*}
where $\Lambda=(\lambda_i)_{i\in\N}$ is our sequence of birth parameters, $\mathcal{P}_{\alpha}$ is the law of a Poisson process on $\R_+$ with intensity $\alpha$ (for $\alpha \in\{1,\gamma,\lambda_\infty\}$), and $\mathcal{U}_{[0,1]}$ is the uniform law on $[0,1]$.

An outcome $\omega=(\omega_e^{\infty},u_e,\omega^1_x,\omega^\gamma_x)_{e,x}\in\Omega$ is a realization of all these Poisson processes. For each vertex $x\in\Z^d$, we obtain Poisson processes of parameters $1$ and $\gamma$ and for each edge $e\in\E^d$, we obtain coupled Poisson processes of parameters $(\lambda_i)_i$ described as follows: if $T^\infty_e=(t_k)_{k\in\N}$ are the arrival times of the process $\omega_e^{\infty}$ following $\mathcal{P}_{\lambda_\infty}$ and $u_e=(u_k)_{k\in\N}$ random variables following ${\mathcal{U}_{[0,1]}}^{\otimes\N}$, then the arrival times of the Poisson process $\omega_e^i$ of parameter $\lambda_i$ are
\begin{align*}
 T^i_e=\left\{t:\exists k\in\N^*\text{ with }t=t_k\text{ and }u_k\leq \frac{\lambda_i}{\lambda_\infty}\right\}.
\end{align*}

From now on, we work with the probability space $(\Omega,\mathcal{F},\P_{\Lambda,\gamma})$. The measure $\P_{\Lambda, \gamma}$ has positive spatiotemporal correlations and satisfies the FKG inequality: for all $L\ge0, T\ge 0$ and for all increasing functions $f$ and $g$ on the configurations on $[-L,L]^d\times[0,T]$ we have
\begin{align}
\label{FKG}
\E_{\Lambda,\gamma}\left[fg\right]\ge\E_{\Lambda,\gamma}\left[f\right]\E_{\Lambda,\gamma}\left[g\right]
\end{align}
where $\E_{\Lambda,\gamma}$ is the expectation under $\P_{\Lambda,\gamma}$. When we look at configurations at fixed time, we talk about positive spatial correlations.

\subsection{Graphical construction of the contact process}
For more details, see Harris~\cite{harris78}. Just as in the contact process situation, the contact process with aging can be constructed graphically using families of independent Poisson processes. To get the graphical representation, we begin with the space-time diagram $\Z^d\times\R^+$. It is augmented to give a percolation diagram as follows. First, we consider the arrival times of independent families of Poisson processes:
\begin{itemize} 
\item For each $x\in \Z^d$, denote by $(U_n^x)_n$ the arrival times of the Poisson process $\omega_x^{1}$ with intensity $1$, encoding the potential deaths at $x$. For each space-time point $(x,U_n^x)$, we put a cross $\times$ to indicate that a death will occur if $x$ is alive.
\item For each $x\in\Z^d$, denote by $(W_n^x)_n$ the arrival times of the Poisson process $\omega_x^{\gamma}$ with intensity $\gamma$, encoding the potential agings at $x$.
For each $(x,W_n^x)$, we put a circle $\circ$ to indicate a maturation of $x$: if $x$ has age $k$ before a circle, then $x$ will have age $k+1$ after the circle. 
\item For each $e=\{x,y\}\in\E^d$ and $k\in\N^{*}$, we recall that $(T^{k}_{\{x,y\}})_n$ are the arrival times of the Poisson process $\omega_e^k$ with rate $\lambda_k$, encoding the potential birth through $\{x,y\}$. Between $(x,T_{\{x,y\}}^k)$ and $(y,T_{\{x,y\}}^k)$ we draw an arrow to indicate that, if $x$ is alive with age at least $k$, then there will be a birth at $y$ (with age $1$) if it is not already alive. 
\end{itemize}
\smallskip
\noindent
For $x,y\in\Z^d$, $s<t$ and $i,k\in\N^{*}$, we say that there is an open path from $(x,i,s)$ to $(y,k,t)$ in the diagram if there is a sequence of times $s=s_0<s_1<\cdots<s_{n+1}=t$ and a corresponding sequence of spatial locations $x=x_0,x_1,\ldots,x_n=y$ such that:
\begin{itemize}
\item For $j=1,\ldots,n$, there is an arrow from $x_{j-1}$ to $x_j$ at time $s_j$ so we set \\ $k_j=\min\{k\in\N^{*},~s_{j}\in T_{\{x_{j-1},x_{j}\}}^k\}$.
\item For $j=0,\ldots,n$, the vertical segment $\{x_j\}\times(s_j,s_{j+1})$ does not contain any cross $\times$.
\item For $j=0,\ldots n-1$, the vertical segment $\{x_{j}\}\times(s_j,s_{j+1})$ contains at least $k_{j}$ circles $\circ$.
\item The vertical segment $\{x_0\}\times(s_0,s_1)$ contains at least $k_1-i$ circles and the vertical segment $\{x_n\}\times(s_n,s_{n+1})$ contains exactly $k$ circles.
\end{itemize}
\smallskip
\noindent
With this considerations in mind, let us turn to the definition of our central process $\xi$. For $x,y\in\Z^d$, $i,k\in\N^{*}$ and $t> 0$, we set $\xi_t^{i\delta_x}(y)=k$ if there exists an open path from $(x,i,0)$ to $(y,k,t)$ but not one from $(x,i,0)$ to $(y,k+1,t)$. We extend the definition to any function $f:\Z^d\rightarrow\N$:
\[
\xi_t^f= \max_{x\in\supp f} \xi_t^{f(x)\delta_x}.
\] 
We now define a useful set: let $A_t^x$ be the set of living points at time $t$ that is 
\[  A_t^x = \supp \xi_t^x =\{y\in\Z^d:~ \xi_t^x(y)\neq0\}.\]
We define $A_t^f$ and $A_t^A$ along the same lines.

By construction, we have a property of attractivity: for all functions $f,g$ from $\Z^d$ to $\N$, $f\leq g\implies \xi_t^f\leq \xi_t^g$ and $A_t^f\leq A_t^g$. Moreover, additivity of the process is also immediate from this construction and from the fact that each transition is additive. So, for all functions $f,g$ from $\Z^d$ to $\N$, $\xi_t^{f\vee g}=\xi_t^f\vee\xi_t^g$.

We can extend Harris' proof to show that $(\xi_t^f)_{t\geq 0}$ satisfies the dynamics previously described. Exactly like the classical contact process, the Contact Process with Aging is a Feller process and in particular, it satisfies the strong Markov property. 

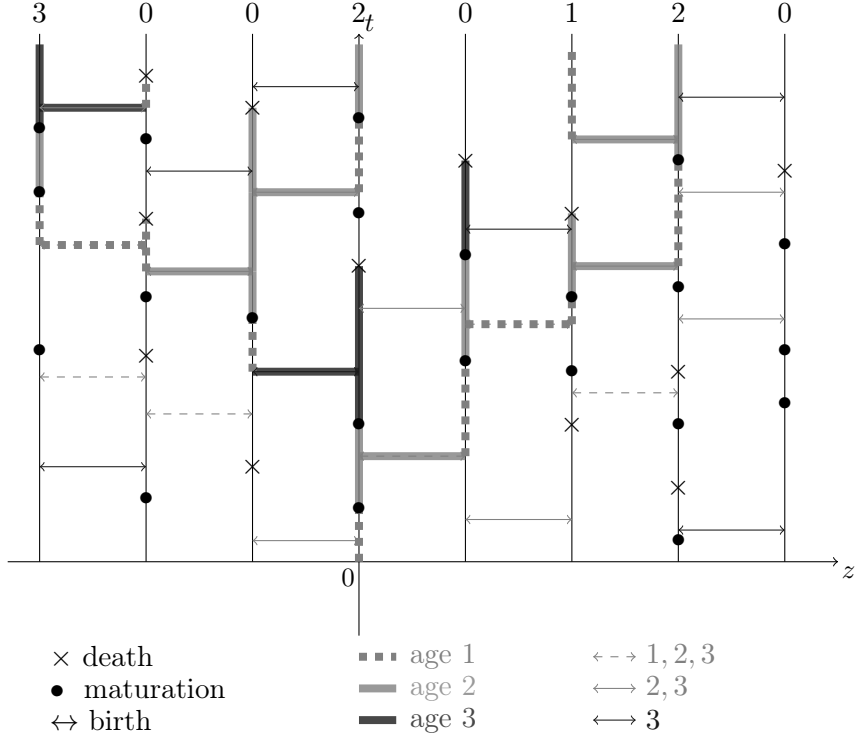
\begin{figure}[!ht]
 \begin{center}

\begin{tikzpicture}[scale=1.4]
\draw[white] (-1,5.5)--(1,5.5);
\draw [->] (-3.3,0)--(4.5,0);
\draw [->] (0,-0.7)--(0,5);
\draw (0.1,5.1) node{\small{$t$}};
\draw (4.6,-0.1) node{\small{$z$}};
\draw (-0.1,-0.15) node{\small{$0$}};
\draw  (1,0)--(1,5); \draw  (2,0)--(2,5); \draw  (3,0)--(3,5); \draw  (4,0)--(4,5);
\draw  (-1,0)--(-1,5); \draw  (-2,0)--(-2,5); \draw  (-3,0)--(-3,5);
\draw[color=gray,dashed,line width=3pt] (0,0)--(0,0.5);
 \draw[gray,line width=3pt,opacity=0.8] (0,0.5)--(0,1.3);
\draw[gray,line width=3pt,opacity=0.8] (0,1)--(1,1);
\draw[color=black!90,line width=3pt,opacity=0.8] (0,1.3)--(0,2.8);
\draw[gray,dashed,line width=3pt] (1,1)--(1,1.9);
\draw[gray,line width=3pt,opacity=0.8] (1,1.9)--(1,2.9);
\draw[gray,dashed,line width=3pt] (1,2.25)--(2,2.25);
\draw[gray,dashed,line width=3pt] (2,2.25)--(2,2.5);
\draw[gray,line width=3pt,opacity=0.8] (2,2.5)--(2,3.3);
\draw[black!90,line width=3pt,opacity=0.8] (1,2.9)--(1,3.8);
\draw[black!90,line width=3pt,opacity=0.8] (0,1.8)--(-1,1.8);
\draw[gray,dashed,line width=3pt] (-1,1.8)--(-1,2.3);
\draw[gray,line width=3pt,opacity=0.8] (-1,2.3)--(-1,2.75);
\draw[gray,line width=3pt,opacity=0.8] (-1,2.75)--(-2,2.75);
\draw[gray,line width=3pt,opacity=0.8] (2,2.8)--(3,2.8);
\draw[gray,dashed,line width=3pt] (3,2.8)--(3,3.8);
\draw[gray,line width=3pt,opacity=0.8] (-1,2.75)--(-1,3.5);
\draw[gray,dashed,line width=3pt] (-2,2.75)--(-2,3.25);
\draw[gray,line width=3pt,opacity=0.8] (-1,3.5)--(0,3.5);
\draw[gray,line width=3pt,opacity=0.8] (-1,3.5)--(-1,4.3);
\draw[gray,dashed,line width=3pt] (-2,3)--(-3,3);
\draw[gray,dashed,line width=3pt] (-3,3)--(-3,3.5);
\draw[gray,line width=3pt,opacity=0.8] (3,3.8)--(3,4.9);
\draw[gray,line width=3pt,opacity=0.8] (3,4)--(2,4);
\draw[gray,dashed,line width=3pt] (2,4)--(2,4.9);
\draw[gray,line width=3pt,opacity=0.8] (-3,3.5)--(-3,4.1);
\draw[black!90,line width=3pt,opacity=0.8] (-3,4.1)--(-3,4.9);
\draw[gray,dashed,line width=3pt] (0,3.5)--(0,4.2);
\draw[gray,line width=3pt,opacity=0.8] (0,4.2)--(0,4.9);
\draw[black!90,line width=3pt,opacity=0.8] (-3,4.3)--(-2,4.3);
\draw[gray,dashed,line width=3pt] (-2,4.3)--(-2,4.6);
\draw(-3,5.2) node{$3$}; \draw(-2,5.2) node{$0$}; \draw(-1,5.2) node{$0$};
\draw(0,5.2) node{$2$}; \draw(1,5.2) node{$0$}; \draw(2,5.2) node{$1$};
\draw(3,5.2) node{$2$}; \draw(4,5.2) node{$0$};
\draw  (-2,4.6) node{$\times$}; \draw  (-2,1.95) node{$\times$}; \draw  (-2,3.25) node{$\times$};
\draw  (-1,0.9) node{$\times$}; \draw  (-1,4.3) node{$\times$}; \draw  (0,2.8) node{$\times$};
\draw  (1,3.8) node{$\times$}; \draw  (2,1.3) node{$\times$}; \draw  (2,3.3) node{$\times$};
\draw  (3,0.7) node{$\times$}; \draw  (3,1.8) node{$\times$}; \draw  (4,3.7) node{$\times$};
\draw  (-3,4.1) node{$\bullet$};  \draw  (-3,2) node{$\bullet$}; \draw  (-3,3.5) node{$\bullet$};
\draw  (-2,0.6) node{$\bullet$}; \draw  (-1,2.3) node{$\bullet$}; \draw  (0,1.3) node{$\bullet$};
\draw  (0,4.2) node{$\bullet$}; \draw  (1,2.9) node{$\bullet$};  \draw  (2,2.5) node{$\bullet$};
\draw  (3,0.2) node{$\bullet$}; \draw  (3,1.3) node{$\bullet$}; \draw  (4,2) node{$\bullet$};
\draw(-2,2.5) node{$\bullet$}; \draw  (0,0.5) node{$\bullet$};
\draw  (0,3.3) node{$\bullet$}; \draw  (2,1.8) node{$\bullet$};
\draw  (1,1.9) node{$\bullet$};  \draw  (3,3.8) node{$\bullet$};
\draw  (3,2.6) node{$\bullet$}; \draw  (4,1.5) node{$\bullet$};
\draw  (4,3) node{$\bullet$}; \draw  (-2,4) node{$\bullet$};
\draw [<->,color=gray!99] (0,0.2)--(-1,0.2);
\draw [<->,gray,dashed] (-3,3)--(-2,3);
\draw [gray,dashed,<->] (-2,1.4)--(-1,1.4);
\draw[gray,dashed,<->] (-3,1.75)--(-2,1.75);
\draw [color=gray!99,<->] (-2,2.75)--(-1,2.75);
\draw [black!90,<->] (-2,3.7)--(-1,3.7);
\draw [black!90,<->] (-1,1.8)--(0,1.8);
\draw [color=gray!99,<->] (-1,3.5)--(0,3.5);
\draw [<->,gray,dashed] (0,1)--(1,1);
\draw [color=gray!99,,<->] (0,2.4)--(1,2.4);
\draw [color=gray!99,<->] (1,0.4)--(2,0.4);
\draw [<->,gray,dashed] (1,2.25)--(2,2.25);
\draw [<->,black!90] (1,3.15)--(2,3.15);
\draw [<->,gray,dashed] (2,1.6)--(3,1.6);
\draw [color=gray!99,<->] (2,2.8)--(3,2.8);
\draw [color=gray!99,<->] (2,4)--(3,4);
\draw [black!90,<->] (3,0.3)--(4,0.3);
\draw [color=gray!99,<->] (3,3.5)--(4,3.5);
\draw [black!90,<->] (-3,0.9)--(-2,0.9);
\draw[color=gray!99,<->] (-3,4.3)--(-2,4.3);
\draw [black!90,<->] (-1,4.5)--(0,4.5);
\draw [color=gray!99,<->] (3,2.3)--(4,2.3);
\draw [black!90,<->] (3,4.4)--(4,4.4);
\draw (-3,-0.9) node[right] {$\times$ death};
\draw (0-3,-1.2) node[right] {$\bullet$ ~maturation}; 
\draw (-3,-1.5) node[right] {$\leftrightarrow$ birth};
\draw[gray,dashed,line width=3pt] (0,-0.9)--(0.35,-0.9) node[right] {age 1};
\draw[gray,line width=3pt,opacity=0.8] (0,-1.2)--(0.35,-1.2) node[right] {age 2};
\draw[black!90,line width=3pt,opacity=0.8] (0,-1.5)--(0.35,-1.5) node[right] {age 3};
\draw[gray,dashed,<->] (2.2,-0.9)--(2.6,-0.9) node[right] {$1,2,3$};
\draw[color=gray!99,<->] (2.2,-1.2)--(2.6,-1.2) node[right] {$2,3$};
\draw[black!90,<->] (2.2,-1.5)--(2.6,-1.5) node[right] {$3$};
\end{tikzpicture}
\end{center}
\caption{Graphical representation of the Contact Process with Aging on $\Z\times\R^+$ (starting from $0$ with age $1$)}
\label{pca}
\end{figure}

\subsection{Time and spatial translations} 
For any $t\geq 0$, we define the time translation operator $\theta_t$ on a locally finite counting measure $m=\sum_{i=1}^\infty \d_{t_i}$ on $\R_+$ by setting 
\[ \theta_t m=\sum_{i=1}^\infty \1_{t_i\geq t}\d_{t_i-t}.\]
It induces an operator on $\Omega$, still denoted $\theta_t$. The Poisson point processes being translation invariant, the probability measure $\P_{\Lambda,\gamma}$ is stationary under the action of $\theta_t$.

Trajectorial version of the semi-group property of the contact process: for every \\$f:\Z^d\rightarrow\N$, for every $s,t\geq0$, for every $\w \in \Omega$, we have
 \[
 \xi_{t+s}^f(\w)=\xi_s^{\xi_t^f(\w)}(\theta_t\w)=\xi_s^{\bullet}(\theta_t\w)\circ \xi_t^f(\w)
 \]
that can also be written in the classical Markovian way: for every $B\in \mathcal{B}({\mathcal{D}})$
\[\P_{\Lambda,\gamma}\left(\left(\xi_{t+s}^f\right)_{s\geq 0}\in B |\mathcal{F}_t\right)= 
\P_{\Lambda,\gamma}\left(\left(\xi_{s}^{\bullet}\right)_{s\geq 0}\in B\right)\circ\xi_t^f.\]
We have the strong Markov property too.
  
We define the spatial translation operator $T_x$ for $x\in\Z^d$ and $\omega\in\Omega$ by \[T_x\omega =\left(\left(\omega_{x+e}\right)_{e\in\E^d},\left(\omega_{x+z}\right)_{z\in\Z^d}\right).\]
$\P_{\Lambda,\gamma}$ is obviously stationary under the action of $T_x$.

Note that $\left(\Omega,\mathcal{F},\P_{\Lambda,\gamma},\theta_t\right)$ and $\left(\Omega,\mathcal{F},\P_{\Lambda,\gamma},T_x\right)$ are both dynamical systems. As strong mixing systems, they are also ergodic.

\subsection{Essential hitting times and associated translations}
For $f:\Z^d\rightarrow\N$, we define the life time of the process starting from $f$ by
\begin{align*}
\tau^f=\inf\{t\geq0: \xi_t^f\equiv0\}= \inf\{t\geq0: A_t^f=\emptyset\}.
\end{align*}
For $f:\Z^d\rightarrow\N$ and $x\in\Z^d$, we also define the hitting time of the site $x$:
\begin{align*}
t^f(x)=\inf\{t\geq 0: \xi_t^f(x)\neq 0\}=\inf\{t\geq 0: x\in A_t^f\}.
\end{align*}
For the sake of clarity, we denote by $\tau^x=\tau^{\delta_x}$, $\tau=\tau^0$ and $t(x)=t^{\delta_0}(x)$.

For fixed $x\in\Z^d$, if we want to prove that the hitting times are such that $t^f(nx)/n$ converges, the Kingman theory requires subadditivity properties. But with non permanent models like the contact process, the hitting times can be infinite (because extinction is possible) and if we condition on the survival, we can lose independence, stationarity and even subadditivity properties.

This is why we rather work with the so-called \textit{essential hitting time} introduced by Garet and Marchand in~\cite{GMPCRE}. For the contact process with aging, we can define it exactly in the same way. We set $u_0(x)=v_0(x)=0$ and we define by induction two sequences of stopping times $(u_n(x))_n$ and $(v_n(x))_n$ as follows.
\begin{itemize}
\item Suppose that $v_k(x)$ is defined. We set 
\begin{align*}
u_{k+1}(x)&=\inf\{t\geq v_k(x):\xi_t^0(x)\neq 0\}\\
&=\inf\{t\geq v_k(x):x\in A_t^0\}.
\end{align*} 
If $v_k(x)$ is finite, then $u_{k+1}(x)$ is the first time after $v_k(x)$ where the site $x$ is once again alive; otherwise, $u_{k+1}(x)=+\infty$.
\item Suppose that $u_k(x)$ is defined, with $k\geq 1$. We set $v_k(x)=u_k(x)+\tau^x\circ \theta_{u_k(x)}$. If $u_k(x)$ is finite, then the time $\tau^x\circ\theta_{u_k(x)}$ is the (possibly infinite) extinction time of the CPA starting at time $u_k(x)$ from the configuration $\delta_x$; otherwise, $v_k(x)=+\infty$.
\end{itemize}
We have $u_0(x)=v_0(x)\leq u_1(x)\leq v_1(x) \leq \ldots \leq u_i(x) \leq v_i(x) \ldots$ 
We then define $K(x)$ to be the first step when $v_k$ or $u_{k+1}$ becomes infinite: 
\[
K(x)=\min\{k\geq 0: v_k(x)=+\infty\text{ or }u_{k+1}(x)=+\infty\}. \]
In Section~\ref{Sergo}, we will prove the following result:
\begin{lemma}
\label{K}
 $K(x)$ has a sub-geometric tail. In particular, $K(x)$ is almost surely finite.
\end{lemma}
\begin{definition}\label{defsigma} We call essential hitting time of $x$ the quantity $\sigma(x)=u_{K(x)}$. 
\end{definition}
The quantity $\sigma(x)$ is a certain time when $x$ is alive and has infinite progeny, but it is not necessary the first such time. 
By Lemma~\ref{K}, $\sigma(x)$ is well-defined. We define the operator $\tilde{\theta}_x$ on $\Omega$ by setting:
\begin{equation}
\label{theta}
\tilde{\theta}_x=\begin{cases}
  T_x\circ\theta_{\sigma(x)} & \text{if } \sigma(x)<+\infty,\\
  T_x &\text{otherwise}.
\end{cases}
\end{equation}
\begin{remark} It is easily checked that $u_k(x)$ and $u_k(-x)$ (respectively $v_k(x)$ and $v_k(-x)$, $K(x)$ and $K(-x)$) are identically distributed. Accordingly, $\sigma(x)$ and $\sigma(-x)$ are identically distributed.
\end{remark}

\subsection{Survival}
Let $\rho(\Lambda,\gamma)=\P_{\Lambda,\g}\left(\forall t>0,~\xi_t\not\equiv 0\right)$ be the probability that the process survives from the point $0$ at age $1$. For $\gamma>0$, we denote by
\[ S_{\g}=\left\{\Lambda\in\R^\N,\text{ non decreasing }\left/\right.\P_{\Lambda,\g}\left(\forall t>0,~\xi_t\not\equiv 0\right)>0\right\}\]
the supercritical region.

We will see, in the short Section~\ref{Ssurvival}, that for $\gamma$ large enough, $S_\gamma$ is not empty. Besides, the supercritical region does not depend on the finite initial configuration. From Section~\ref{Scontrols} to the end, we will suppose that $\Lambda\in S_\gamma$. We will say that the process $(\xi_t)$ is supercritical when it survives with positive probability. In this context, we will work with a CPA conditioned to survive and we introduce the corresponding probability measure:
\[\forall E \in\mathcal{F},~\overline{\P}_{\Lambda,\g}(E)=\P_{\Lambda,\g}\left(E|~\forall t>0,~\xi_t\not\equiv 0\right).\]
When the context is clear we will omit the indices $\Lambda,\gamma$ in the above notations $\P$ and $\overline{\P}$.

\subsection{Organization of the paper}
First, in Section~\ref{Ssurvival}, we prove the above announced properties on the supercritical region. Section~\ref{Sbezgrim} is devoted to the key point of our work: the construction of a background percolation process when the CPA is supercritical (Theorems~\ref{FST} and~\ref{construction}). In Section~\ref{Scontrols}, we use the construction to prove the following crucial exponential estimates:
\begin{restatable}{theorem}{thmpetitsamas}
\label{petitsamas} For $f:\Z^d\to\N$, if $(\xi_t^f)$ is supercritical, then there exist $A,B>0$ such that for all $t>0$ and $x\in\Z^d$, one has
\begin{align*}
\P\big(t<\tau^f<\infty\big)&\leq A\exp(-B t).
\end{align*}
\end{restatable}
\begin{restatable}{theorem}{thmaumlineaire}
\label{au-lineaire} For $f:\Z^d\to\N$, if $(\xi_t^f)$ is supercritical, then there exist $A,B>0$ such that for all $t>0$ and $x\in\Z^d$, one has
\begin{align*}
\P(t^{f}(x)\geq C\|x\|+t,\tau^{f}=\infty)&\leq A\exp(-Bt),
\end{align*}
which we call the ``at least linear growth''.
\end{restatable}
These controls enable us to establish in Section~\ref{Sergo} the main properties of the transformations $\tilde{\theta}_x$:

\begin{theorem}\label{ergo} For $x\in\Z^{d}\setminus\{0\}$,
\begin{itemize}
\item The probability measure $\Pbarre$ is invariant under the action $\tilde \theta_x$.
\item Under $\Pbarre$, $\sigma(y)\circ\tilde{\theta}_x$ is independent from $\sigma(x)$ and its law is the same as the law of $\sigma(y)$.
\end{itemize}
\end{theorem}
We also require the exponential estimates to establish in Section~\ref{Sssadd} the following almost subadditivity property of the essential hitting time $\sigma$: 
\begin{restatable}{theorem}{thmpresquesousadditif}
\label{presquesousadditif}
If $(\xi_t)$ is supercritical, then there exist $A,B>0$ such that for all $x,y\in\Z^d$ and $t>0$, one has
\[\Pbarre\left(\sigma(x+y)-\left(\sigma(x)+\sigma(y)\circ
\tilde{\theta}_x\right)\ge t\right)\le A\exp\left(-B\sqrt{t}\right).\]
\end{restatable}
We also provide, in Section~\ref{Sssadd}, a control for the difference between the hitting time $t$ and the essential hitting time $\sigma$:  
\begin{restatable}{theorem}{thmdifferenceuniforme}
\label{differenceuniforme}
If $(\xi_t)$ is supercritical, then, $\Pbarre$ almost surely, it holds that \[\lim_{\|x\|\to+\infty} \frac{|\sigma(x)-t(x)|}{\|x\|}=0.\]
\end{restatable}

Finally, in Section~\ref{STFA}, we prove the expected asymptotic shape theorem (thanks to Theorems~\ref{au-lineaire}, \ref{ergo}, \ref{presquesousadditif}, \ref{differenceuniforme} and thanks to a result by Kesten and Hammersley):

\begin{restatable}{theorem}{thmTFA}
 \label{TFA}
If $(\xi_t)$ is supercritical, then there exists a norm $\mu$ on $\R^{d}$ such that for every $\epsilon>0$, almost surely under $\Pbarre$, for every large $t$
\[ (1-\epsilon)B_{\mu}\subset\frac{\tilde H_t}t\subset(1+\epsilon)B_{\mu}\] 
 $\tilde{H}_t=\{x\in\Z^d:t(x)\le t\}+[0,1]^d$ and $B_\mu$ is the unit ball for $\mu$. 
\end{restatable}

\section{About the survival}\label{Ssurvival}

For the classical contact process, Harris defined the critical value:
\[\l_c=\inf\left\{\lambda\ge0,~\P_\l\left(\forall t>0,~\xi_t^{\{0\}}\neq\emptyset\right)>0\right\} \]
and proved that $\l_c\in(0,+\infty)$.
In his model, Krone defined a similar value:
\[ \l_c(\g)=\inf\left\{\lambda\geq0,~\P_{\l,\g}\left(\forall t>0,~\xi_t^{\{0(2)\}}\neq\emptyset\right)>0\right\} \]
where $0(2)$ means the site $0$ in the adult state (state $2$). He showed that if $\g$ and $\l$ are sufficiently large, then $\P_{\l,\g}(\forall t>0,~\xi_t^{\{0(2)\}}\neq\emptyset )$ is positive. Therefore, $\l_c(\g)$ is not trivial for $\g$ sufficiently large. 
He drew the look of the curve $\l_c(\g)$ and obtained a partition of $(\R^+)^2$ in two regions: survival and extinction.

We generalize this for our model. Recall that $\rho(\Lambda,\gamma)=\P_{\Lambda,\g}\left( \forall t>0,~\xi_t\not\equiv 0\right)$ and
\[ S_{\g}=\left\{\Lambda\in\R^\N,\text{ non decreasing }\left/\right.\P_{\Lambda,\g}\left(\forall t>0,~\xi_t\not\equiv 0\right)>0\right\}\]
the supercritical region.

We also recall that if it is not specified, the initial set of our process is the site $0$ with age $1$. The aspects of this region and its frontier depend on the chosen topology. Let us make some obvious remarks about it. Let $\Lambda=(\lambda_i)_i$ and $\Lambda'=(\l_i')_i$; if $\Lambda\in{S_\g}$ and for all $i$, $\l_i\leq\l_i'$ then $\Lambda'\in{S_\g}$. If we take for all $i$, $\lambda_i> \lambda_c$, then the process is supercritical; but this assumption is too strong for our model. We want to say something more relevant so we show:

\begin{proposition} For every $m\in\N$, for fixed $\lambda_1,\ldots,\lambda_m$, we can find $\lambda_{m+1}$ and $\gamma$ large enough such that $\rho(\Lambda,\gamma)>0$ with $\Lambda=(\lambda_1,\ldots,\lambda_{m+1},\lambda_{m+1},\ldots)$. \end{proposition}
\begin{proof} We start the process with the point $0$ with age one. Without lost of generality, we assume that $\lambda_1=\cdots=\lambda_m=0$.

We use a comparison between our model of contact process with aging and a 1-dependent oriented site percolation (on $\Z\times\N)$. The construction is similar to Harris' proof of the survival for the contact process when $\l$ is large enough (\cite{harris74}). Let $T>0$ and $n\in\N$. We say that $x$ is good at time $nT$ if
\begin{enumerate}
 \item \label{nodeath} there is no death at $x$ during $[nT, nT+\frac{3T}{2}]$,
 \item \label{mature} there are enough maturations (more than $m$) at $x$ during $[nT+\frac{T}{2}, nT+T]$,
 \item \label{birth} there are arrows from $x$ to each of its neighbors on $[nT+T, nT+\frac{3T}{2}]$.
\end{enumerate}
At time $nT+T$ we do not know the age of $x$. If we wait $m$ maturations, we are sure that it is possible to use the birth arrows of rate $\lambda_{m+1}$ to give birth to the neighbors.

For $\alpha\in\{1,\gamma,\lambda_{m+1}\}$, we denote by $P^\alpha$ a Poisson process of parameter $\alpha$. One has
\begin{align*}
\P(x\text{ is good at }nT)&=\P(\text{\ref{nodeath}, \ref{mature} and \ref{birth} are satisfied})\\
 &=\P\left(P^{1}\big([nT, nT+\frac{3T}{2}]\big)= 0\right)\times \P\left(P^\gamma\big([nT+\frac{T}{2}, nT+T]\big)\geq m\right)\\
 &\quad\times\P\left(P^{\lambda_{m+1}}\big([nT+T, nT+\frac{3T}{2}]\big)\neq 0\right)^{2d}\\
 &\geq \exp(-\frac{3T}{2})\P\left(P^\gamma\big([0, \frac{T}{2m}\big)\neq 0\right)^m\big(1-\exp\big(-\l_{m+1}\frac{T}{2}\big)\big)^{2d}\\
  &\geq \exp(-\frac{3T}{2})\left(1-\exp\big(-\gamma\frac{T}{2m}\big)\right)^m \big(1-\exp\big(-\l_{m+1}\frac{T}{2}\big)\big)^{2d}.
\end{align*}
We pick $T$ small enough such that $\exp(-\frac{3T}{2})>1-\epsilon$. We pick $\gamma$ large enough such that $\left(1-\exp\left(-\gamma\frac{T}{2m}\right)\right)^m\ge 1-\epsilon$, and $\l_{m+1}$ large enough such that $1-\exp\left(-\l_{m+1}\frac{T}{2}\right)\geq 1-\epsilon$.
Then
\[\P(x\text{ is good at }nT)\geq 1-3\epsilon.\]
Let $\mathcal{L}=\{(m,n)\in\Z\times\N, m+n \text{ is even}\}$ be the lattice of oriented percolation. We say that $(m,n)\in\mathcal{L}$ is open if the site $(m,0,\ldots,0)\in\Z^d$ is good at time $nT$. If $\epsilon$ is sufficiently small, the probability $p$ that the site is open is sufficiently large and we have percolation in the lattice $\mathcal{L}$ (by a contour argument, see~\cite{durrettperco} for details). The construction implies that $(\xi_t)$ survives with positive probability.
\end{proof}

Now, we prove that survival does not depend on the initial function (with finite support).
\begin{proposition}
 Let $f,f':\Z^d\rightarrow\N$ be functions with non empty finite supports. One has the equivalence
 \[\P_{\Lambda,\g}\left(\forall t>0,~\xi_t^f\not\equiv 0 \right)>0\Leftrightarrow \P_{\Lambda,\g}\left(\forall t>0,~\xi_t^{f'}\not\equiv 0 \right)>0.\]
\end{proposition}

\begin{proof}
First, if we start with only one living point, then let us see that the survival does not depend on its age. Let $x\in\Z^d$ and $n\in\N^*$. By construction,
\[
\P\big(\forall t>0,~\xi_t^{n\delta_x}\not\equiv 0 \big)\leq \P\left(\forall t>0,~\xi_t^{(n+1)\delta_{x}}\not\equiv 0 \right).
\]

Conversely, let $T$ be the time of the first maturation on $x$. One has
\begin{align*}
\P_{\Lambda,\g}\big(\forall t>0,~ \xi_t^{n\delta_x}\not\equiv 0 \big) &\geq \P\left(\left\{\forall t\in[0,T],~\xi_t^{n\delta_x}(x)\neq 0 \right\}\cap \left\{\theta_T\left(\forall t>0,~\xi_t^{(n+1)\delta_x}\right)\not\equiv 0  \right)\right\}\\
&\geq \P_{\Lambda,\g}\left(\exp(\delta)\leq \exp(1)\right)\P_{\Lambda,\g}\left(\theta_T\left( \forall t>0,~\xi_t^{(n+1)\delta_x}\not\equiv 0\right)\right)\\
&\geq \frac{\d}{1+\d} \P_{\Lambda,\g}\left(\forall t>0,~\xi_t^{(n+1)\delta_x}\not\equiv 0 \right).
\end{align*}
Secondly, the survival does not depend on the finite number of initial living points. Indeed, for all $x_1,\ldots,x_n\in\Z^d$ and $m_1,\ldots,m_n\in\N^*$, one has
\begin{align*}
\P\left(\forall t>0,~\xi_t^{\sum_{i=1}^n m_i\delta_{x_i}}\not\equiv 0 \right)&\leq \sum_{i=1}^n  \P\left(\forall t>0,~\xi_t^{m_i\delta_{x_i}}\not\equiv 0 \right)~\text{ by additivity},\\
& \leq  n\P\left(\forall t>0,~\xi_t^{\max_i\{m_i\}\delta_{x_{\arg\max\{m_i\}}}}\not\equiv 0 \right),\\
\text{and }~~\P\left(\forall t>0,~\xi_t^{\sum_{i=1}^n m_i\delta_{x_i}}\not\equiv 0 \right)&\geq  \P\left(\forall t>0,~\xi_t^{m_1\delta_{x_1}}\not\equiv 0 \right)\text{ by attractivity}.\qedhere
\end{align*}
\end{proof}

Note that if we have survival, then $\rho(\Lambda,\gamma)>0$ and for every $f:\Z^d\to\N$ not identically null we have, by monotonicity:
\begin{align}
\label{survival}\P_{\Lambda,\g}\left(\forall t>0,~\xi_t^f\not\equiv 0\right)\geq \rho(\Lambda,\gamma)>0.\end{align}
Whenever the context is clear ($\Lambda,\gamma$ fixed) we will just write $\rho=\rho(\Lambda,\gamma)$.

\section{Construction of a percolation background process}\label{Sbezgrim}
To obtain the exponential bounds of Theorems~\ref{petitsamas} and~\ref{au-lineaire}, we construct a background percolation process for the supercritical contact process with aging, just as Bezuidenhout and Grimmett did in~\cite{bezgrim} for the classical contact process (we also refer to~\cite{lig99} and~\cite{steifwar}). We work with the contact process with aging $(\xi_t)$, and also with its support $(A_t)$. The process $(A_t)$ looks like the classical contact process (it takes values in $\{0,1\}^\N$ and encodes the fact that a given particle is alive or not) but it is not Markovian. We formulate a lot of propositions about the process $(A_t)$ when we only need to know the number of living particles regardless their age. However, the Markovian aspect of $(\xi_t)$ is essential in the different proofs. In this section, we fix $\Lambda$ and $\gamma$ and we write $\P$ for $\P_{\Lambda, \gamma}$. We show 
\begin{theorem}
\label{FST}
The supercriticality of the process $(\xi_t)$ is equivalent to the following {\bf finite space-time condition}:  
\begin{align*}
&\forall\epsilon>0,\exists (n,a,b)\text{ such that if }(x,s)\in[-a,a]^d\times[0,b]\text{ then}\\
\label{B1}\tag{B1}
&\P\left( \begin{array}{c}
          \exists(y,t)\in[a,3a]\times[-a,a]^{d-1}\times[5b,6b]\text{ and open paths staying}\\
           \text{in }[-5a,5a]^d\times[0,6b]\text{ and going from }(x,s)+[-n,n]^d\times\{0\}\\
           \text{ to every point in } (y,t)+[-n,n]^d\times\{0\}
         \end{array}
\right)> 1-\epsilon.
\end{align*}
\end{theorem}

\subsection{Supercritical implies finite space-time condition}

Here we suppose that the process is supercritical and we construct a large space-time box with a lot of well-located living particles on the frontier.\\

\textbf{First step:} If $(\xi_t)$ survives with positive probability, we can increase the number of living particles (as well as their age) in the initial configuration so as to ensure that survival occurs with a probability close to one.

\begin{lemma} 
\label{survie} Assume that $(\xi_t)$ is supercritical. Let $(f_n)_n$ be an increasing sequence of functions from $\Z^d$ to $\N$ which converges to $f:\Z^d\to\N$ not identically null and periodic. Then
\[ 
\lim_{n\rightarrow\infty}\P\left(\forall t>0,~\xi_t^{f_n}\not\equiv 0\right)=1.
\] 
\end{lemma}

\begin{proof}
As $(f_n)$ increases to $f$, one has
\[ \lim_{n\rightarrow\infty}\P\left(\forall t,~\xi_t^{f_n}\not\equiv 0\right)=\P\left(\forall t,~\xi_t^{f}\not\equiv 0\right)\geq\P\left(\forall t,~\xi_t \not\equiv 0\right)>0.\]
Given that the system $(\Omega,\mathcal{F},\P,T_x)$ is ergodic and $f$ is periodic, we can conclude that
\[ \P\left(\forall t,~\xi_t^{f}\not\equiv 0\right)>0\Rightarrow\P\left(\forall t,~\xi_t^{f}\not\equiv 0\right)=1.\qedhere \]
\end{proof}

When we are interested in controling the amount of living particles, we will not take their ages into account. Therefore, we will work with $A_t^f=\supp \xi_t^f$.

For $L\geq 1$, we consider the truncated process $\left(\leftidx{_L}{\xi}{_t}\right)$ defined via the graphical representation: to determine $\leftidx{_L}{\xi}{_t}$ we only use paths with vertical segments above sites in $(-L,L)^d$ and horizontal segments (birth arrows) whose origins are in $(-L,L)^d$. $\leftidx{_L}{\xi}{_t}$ takes values in $[-L,L]^d$ but is measurable with respect to the Poisson processes in $(-L,L)^d\times[0,t]$. We have $\supp \leftidx{_L}{\xi}{_t}=\leftidx{_L}{A}{_t}$.

Then we define $\leftidx{_L}{A}{}=\bigcup_{t\geq 0}\left(\leftidx{_L}{A}{_t}\times\{t\}\right)\subset\Z^d\times\left[\right.0,\infty\left.\right)$ as the set of all living space-time points in the truncated process without any consideration of age.\\

\textbf{ Second step:} The probability of survival can be seen as the limit of the probability to have \textit{enough living particles in a finite space time box}.

\begin{lemma} 
\label{etape1}
For every $f:\Z^d\rightarrow\N$ with finite support and every $N\geq 1$, it holds that 
\[ \lim_{t\rightarrow\infty}\lim_{L\rightarrow\infty} \P\left(|\leftidx{_L}{A}{_t^f}|\geq N\right)=\P\left(\forall t>0,~\xi_t^{f}\not\equiv 0\right).\]
\end{lemma}

\begin{proof} 
Since $A_t^f=\bigcup_{L\geq 1}\leftidx{_L}{A}{^f_t}$, it follows that 
\[\lim_{L\rightarrow \infty} \P\left(|\leftidx{_L}{A}{^f_t}|\geq N\right)=\P\left(|A_t^f|\geq N\right).\] 
Let us show that the right-hand side converges to the probability of survival when $t$ goes to infinity. Let $s\in\R^+$, 
\begin{align*}
\P\left(A_t^f=\emptyset\text{ for some }t|\mathcal{F}_s\right)&\geq \P\left(\forall x\in A_s^f, x\text{ dies before giving birth or aging}|\mathcal{F}_s\right).
\end{align*}
Denote by $Y_1^x$ the waiting time until the next death on $x$, $Y_\gamma^x$ the waiting time until the next maturation and $Y_{\xi^f_s(x)}^{x,i}$ the waiting time until the next birth from $x$ to its $i^{th}$ neighbor (where $\xi^f_s(x)$ is the age of $x$ at time $s$). These random variables are independent and follow exponential laws with respective parameters $1,\gamma,\xi^f_s(x)$. One has
\begin{align*}
 \P\left(A_t^f=\emptyset\text{ for some }t|\mathcal{F}_s\right)
&\geq \prod_{x\in A_s^f}\P\left(Y_1^x\leq \min(Y_\g^x,Y^{x,1}_{\xi^f_s(x)},\ldots,Y^{x,2d}_{\xi^f_s(x)})\right)\\
&\geq \prod_{x\in A_s^f}\frac{1}{1+\g+2d\l_{\xi^f_s(x)}}\\
& \geq \left(\frac{1}{1+\g+2d\l_{\infty}}\right)^{|A_s^f|}.
\end{align*} 
By the martingale convergence theorem, the probability $\P(A_t^f=\emptyset\text{ for some }t|\mathcal{F}_s)$ converges almost surely to $\1_{\{A_t^f=\emptyset\text{ for some }t\}}$. It follows that
\[ \forall s,~A_s^f\neq\emptyset\Rightarrow \lim_{s\rightarrow\infty} |A_s^f|=\infty.\]
In other words, if the process survives, then the number of living particles goes to infinity.

\begin{align*}
\P\left(|A_t^f|\geq N\right)&= \P\left(\{|A_t^f|\geq N\}\cap\{\forall s,~A_s^f\neq\emptyset\}\right)\\
&\phantom{={}}+\P\left(\{|A_t^f|\geq N\}\cap\{\exists s,~A_s^f=\emptyset\}\right).\\
\text{So, }~\lim_{t\rightarrow\infty}\P\left(|A_t^f|\geq N\right)&= \P\left(\forall s,~A_s^f\neq\emptyset,\forall s\right).
\end{align*}
\end{proof}

\textbf{Third step:} At fixed time, we compare the number of living particles in a spatial orthant of the top of a large box according to the number of living particles on the top of this box.

From now on, we will need some symmetries assumptions on the initial configurations.
\begin{definition} An initial configuration $f:\Z^d\to\N$ is said to be acceptable if its support is finite and if
 \[\forall (x_1,\ldots,x_d),\forall i\in\{1,\ldots,d\}, f(x_1,\ldots,x_i,\ldots,x_d)=f(x_1,\ldots,-x_i,\ldots,x_d).\]
\end{definition}

\begin{lemma}
\label{FKG1}
For every $N,L\ge 0$, $t>0$ and $f$ acceptable such that $\supp f\subset(-L,L)^d$, one has
\[ \P\left(|\leftidx{_L}{A}{_t^{f}}\cap [0,L)^d|\leq N\right)\leq \P\left(|\leftidx{_L}{A}{_t^f}|\leq 2^d N\right)^{2^{-d}}. \] 
\end{lemma}
\begin{proof} Let $X_1,\ldots,X_{2^d}$ be the number of living particles in the $2^d$ different orthants of $(-L,L)^d$ at time $t$. For instance, $X_1=|\leftidx{_L}{A}{_t^{f}}\cap[0,L)^d|$. We have
\[|\leftidx{_L}{A}{_t^f}|\leq X_1+\cdots+X_{2^d},\]
so
\begin{align*}
\P\left(|\leftidx{_L}{A}{_t^f}|\leq 2^dN\right)&\geq \P\left( X_1+\cdots+X_{2^d}\leq 2^dN\right),\\
&\geq \P\left(\forall i\in \{1,\ldots,2^d\}, X_i\leq N\right).
\end{align*}
The variables $X_1,\ldots,X_{2^d}$ are identically distributed and are increasing functions of the configurations on $(-L,L)^d\times\{t\}$; therefore, by the spatial positive correlations of the measure $\P$, the FKG inequality applies and we obtain that
\[\P\left(|\leftidx{_L}{A}{_t^f}|\leq 2^dN\right)\geq \Big(\P\left(X_1\leq N\right)\Big)^{2^d}.\qedhere\]
\end{proof}

\textbf{Fourth step:} The probability of extinction can be seen as the limit of the probability to \textit{not have enough living particles on the top and the sides of a finite space-time box}.

Let $S(L,T)=\{(x,s)\in\Z^d\times [0,T]:\|x\|_{\infty}=L\}$ be the union of the lateral sides of the box $[-L,L]^d\times [0,T]$. Let $N^f(L,T)$ be the maximal number of points in $S(L,T)\cap\leftidx{_L}{A}{^f}$ satisfying the following property: if $(x,t)$ and $(x,s)$ are any distinct points with the same spatial coordinate, then $|t-s|\geq 1$.

\begin{lemma}
\label{etape2}
 Let $(L_j)$ and $(T_j)$ be two increasing sequences going to infinity. For any $M,N>0$ and $f:\Z^d\to\N$ with finite support,
 \begin{align*}
\limsup_{j\rightarrow\infty} \P\left(N^f(L_j,T_j)\leq M\right)\P\left(|\leftidx{_{L_j}}{A}{^f_{T_j}}|\leq N\right)\leq\P\left(\exists s,~\xi_s^f\equiv 0\right).
\end{align*}
 \end{lemma}
 
We proceed as in Lemma~\ref{etape1} but this time we do it in a space-time box. Here, we also have to control the exit of the process through the lateral sides of the box. 

\begin{proof} Let $\mathcal{F}_{L,T}$ be the $\sigma$-algebra generated by the Poisson processes from the graphical representation in the box $(-L,L)^d\times [0,T]$. 
We first prove that, almost surely on the event $\{N^f(L,T)+|\leftidx{_L}{A}{_T^f}|\leq k\}$, we have
 \[
 \P\left(A_s^f=\emptyset \text{ for some }s>T|\mathcal{F}_{L,T}\right)\geq \left(\frac{e^{-4\l_\infty d}}{1+\gamma+2d\l_{\infty}}\right)^k. 
 \]
 We suppose that $N^f(L,T)+|\leftidx{_L}{A}{_T^f}|\leq k$. By definition of $\leftidx{_L}{A}{^f}$, this event is $\mathcal{F}_{L,T}$-measurable. For $x\in\leftidx{_L}{A}{_T^f}$, we have seen that the probability that $x$ dies before giving birth or aging is at least \[\frac{1}{1+2d\l_{\infty}+\gamma}.\]
For the lateral sides of the box, we will work with the time lines $\{x\}\times[0,T]$ where $\|x\|_\infty=L$. Let $j_x$ be the maximal number of points on this time line in $\leftidx{_L}{A}{^f}$ with the property that each pair is separated by a distance of at least 1; and let $(x,t_1),\ldots,(x,t_{j_x})$ be such a set. Let \[I=\bigcup_{i=1}^{j_x} \{x\}\times (t_i-1,t_i+1).\]
The probability of the event $E_I=\{\text{there is no birth arrow coming out from }I\}$ is at least $e^{-4d\l_{\infty} j_x}$ because the length of $I$ is at most $2j_x$.

Furthermore, let $I_1^c,\ldots,I_{j_x}^c$ be the (possibly empty) intervals in the complement of $I$ in this line. For $1\leq l\leq j_x$, if $I_l^c$ has length $u$, then we consider
\begin{align*}
E_l &=\{\text{there is no birth arrow from $I_l^c$}\}\\
 &\phantom{={}}\cup\{\text{there is at least one birth but there is a death before in $I_l^c$}\}.\\
\text{We have: }\\ \P(E_l|\mathcal{F}_{L,T})&\ge\P\left(P^{\l_{\infty}}([0,u])=0\right)^{2d}+\P\left(Y_1^x\leq \min(Y^x_{2d\l_\infty}, u)\right)\\
&\phantom{={}}~\text{ with the same notations as previously}\\
       &= e^{-2d\l_{\infty} u}+ \int_0^u 2d\l_{\infty} e^{-2d\l_{\infty} s}(1-e^{-s})ds\\
       &= 1-\int_0^u 2d\l_{\infty} e^{-(1+2d\l_{\infty}) s}ds\geq \frac{1}{1+2d\l_{\infty}}.
\end{align*}
The events $(E_1,\ldots,E_j,E_I)$ are independent because they refer to disjoint parts of the graphical representation. Consequently, the probability that none of the points on the time line contributes to the survival of the process is at least 
\begin{align*}
\P(\text{no birth from the time line }\{x\}\times [0,L]|\mathcal{F}_{L,T})& \geq \P(E_I|\mathcal{F}_{L,T})\prod_{1\leq l\leq {j_x}} \P(E_l|\mathcal{F}_{L,T})\\
& \geq  e^{-4d\l_{\infty} j_x} \left(\frac{1}{1+2d\l_{\infty}}\right)^{j_x}. 
\end{align*}
Finally, we have
\begin{align*}
  \P\left(A_s^f=\emptyset \text{ for some }s|\mathcal{F}_{L,T}\right) &\geq 
   \P\left(\forall x\in\leftidx{_L}{A}{_T^f}, x \text{ dies before giving birth or aging}|\mathcal{F}_{L,T}\right)\\
    & \phantom{\ge{}}\times\P\Big(\text{For every lateral time line, no birth from it}|\mathcal{F}_{L,T}\Big)\\
    &\geq \left(\frac{1}{1+\g+2d\l_{\infty}}\right)^{|\leftidx{_L}{A}{_T^f}|} \prod_{x:\|x\|_\infty=L} \left(\frac{e^{-4d\l_{\infty}}}{1+2d\l_{\infty}}\right)^{j_x}\\
&\geq \left(\frac{e^{-4d\l_\infty}}{1+\g+2d\l_{\infty}}\right)^k,
\end{align*}
on $\{N^f(L,T)+|\leftidx{_L}{A}{_T^f}|\leq k\}$.
As before, we deduce that
\[\limsup_{j\rightarrow\infty}\P\left(N^f(L_j,T_j)+|\leftidx{_{L_j}}{A}{_{T_j}^f}|\leq k\right)=\P\left(\exists s,~\xi_s^f\equiv 0\right).\]
The random variables $N^f(L_j,T_j)$ and $|\leftidx{_{L_j}}{A}{_{T_j}^f}|$ are increasing functions of the configurations on $[-L_j,L_j]^d\times[0,T_j]$. Therefore, by the spatiotemporal positive correlations of the measure $\P$, the FKG inequality applies and for every $M,N$ we have
\begin{align*}
\P\left(N^f(L_j,T_j)+|\leftidx{_{L_j}}{A}{_{T_j}^f}|\leq M+N\right)
&\geq\P\left(N^f(L_j,T_j)\leq M\text{ and }|\leftidx{_{L_j}}{A}{_{T_j}^f}|\leq N\right)\\
&\geq\P\left(N^f(L_j,T_j)\leq M\right)\P\left(|\leftidx{_{L_j}}{A}{_{T_j}^f}|\leq N\right).
\end{align*}
We thus obtain the announced result.
\end{proof}

\textbf{Fifth step:} We control the number of living points in an orthant of the lateral sides of a large space-time box according to the total number of living points in its lateral sides.

Let $B_+(L,T)=\{(x,t)\in\Z^d\times[0,T]\text{ such that }x_1=+L, x_i\geq 0\text{ for }2\leq i\leq d\}$; $B_+(L,T)$ is one of the $2d2^{d-1}$ orthants of the lateral sides of the box. And let $N_+^f(L,T)$ be the maximal number of points in $S_+(L,T)\cap\leftidx{_L}{A}{^f}$  with the following property: if $(x,t)$ and $(x,s)$ are any distinct points with the same spatial coordinate, then $|t-s|\geq 1$.
\begin{lemma}
\label{FKG2}
For $L,M,T$ and $f$ acceptable with $\supp f \in (-L,L)^d$,
 \[\P\left(N_+^f(L,T)\leq M\right)^{d2^d}\leq \P\left(N^f(L,T)\leq Md2^d\right).\]
\end{lemma}

\begin{proof}
Let $X_1,\ldots,X_{d2^d}$ be the number of living points in the different orthants of all lateral faces of the space-time box. For example  $X_1=N_+^f(L,T)$. These variables are identically distributed and are increasing functions of the configurations on $[-L,L]^d\times[0,t]$. Therefore, by the spatiotemporal correlations of the measure $\P$, the FKG inequality applies and we obtain the results in the same way as in Lemma~\ref{FKG1}.
\end{proof}

\textbf{Sixth step:} Now we are ready to construct good events on a finite space-time box, under assumption of supercriticality.

\begin{theorem}
\label{C1etC2}
Assume that $(\xi_t)$ is supercritical. Let $(f_n)_n$ be an increasing sequence of acceptable functions which converges to $f:\Z^d\to\N$ not identically null and periodic. For every $\epsilon >0$, there exist choices of $n,L,T$ such that
\begin{align*}
\label{C1}\tag{C1}
 \P\left(\exists x\in [0,L)^d \text{ such that } \leftidx{_{L+2n}}{\xi}{_{T+1}^{f_n}} \ge T_{x}\circ f_n \right)> (1-\epsilon)\text{ and}
 \end{align*}
 \[ \P\left(\begin{array}{c}        
\exists x\in \{L+n\}\times[0,L)^{d-1} \text{ and } t\in[0,T)\\
   \label{C2}\tag{C2}
  \text{ such that } \leftidx{_{L+2n}}{\xi}{_{t+1}^{f_n}} \ge T_{x}\circ f_n
            \end{array}\right)> (1-\epsilon)   .
 \]
\end{theorem}
Without lost of generality, we assume that $\supp f_n=[-n,n]^d$. In order to get an idea, one may think of $f_n$ as $\1_{[-n,n]^d}$.

\begin{proof} 
Combining the previous propositions, we want to construct a large space-time box with enough living points on an orthant of its boundary. If there are enough living points, then at least one of them will spawn a good configuration, within one unit of time, which allows us to restart the process.

\begin{enumerate}
 \item Choice of $n$: We want to start with a good enough configuration ($f_n$) to survive with high probability. Let $\epsilon>0$; using Lemma~\ref{survie} we can choose $n$ large enough so that
\[\P(\forall t>0,~ \xi_t^{f_n}\not\equiv 0)>1-\epsilon^2.\]

\item Choice of $N$: We want the process to contain enough particles ($N$) on an orthant of the top of a space-time box so that at least one of them spawn $f_n$ in one unit of time. Let $x\in\Z^d$ with age $k$; the probability that $x$ succeeds in spawning $f_n$ in one unit of time is at least $\P(T_x\circ(\leftidx{_{n+1}}{\xi}{_1^{k\d_0}} \ge f_n))$. This quantity is bounded from below by $\alpha>0$, independently of $x$ and $k$. We thus choose $N'$ large enough so that $N'$ independent trials will contain at least one success with probability at least $1-\epsilon$, i.e:
\[N' \text{ such that }1-(1-\alpha)^{N'}\geq \epsilon.\] 
Next we choose $N$ large enough to ensure that any $N$ points in $\Z^d$ contain a subset with at least $N'$ points separated from each other by a distance of at least $2n+1$. In this way, the trials will be independent.

\item Choice of $L_j,T_j$: We want at least $N$ living particles on the orthant $\{T\}\times [0,L)^d$ of the top face of the box so we want to choose $L,T$ such that 
\[\P\left(| \leftidx{_{L}}{A}{_T^{f_n}}\cap [0,L)^d|>N\right)\geq 1-\epsilon.\] To use Lemma~\ref{FKG1}, we have to control the number of living particles in the whole top face. By Lemma~\ref{etape1} there exist sequences $(L_j)$, $(T_j)$ increasing to infinity such that, for each $j$
\[\P\left(|\leftidx{_{L_j}}{A}{_{T_j}^{f_n}}|>2^dN\right)=1-\epsilon.\]

\item[(2bis)] Choice of $M$: To satisfy \eqref{C2} we want enough living points on an orthant of the sides of a space-time box so that at least one of them spawns $f_n$ in one unit of time. We thus choose $M'$ large enough so that $M'$ independent trials of the event $\left\{T_x\circ \left(\leftidx{_{n+1}}{\xi}{_1^{k\d_0}} \ge T_{-ne_1} \circ f_n\right)\right\}$, whose probability is bounded from below by $\beta>0$ independent of $x$ and $k$, will contain at least one success with probability at least $1-\epsilon$:
\[M' \text{ such that }1-(1-\beta)^{M'}\geq \epsilon,\]
where $\leftidx{_{n+1}}{\xi}$ designates here the process restricted to the box $[0,2n]\times[-n,n]^d$.
Next we choose $M$ large enough so that any $M$ points in $\Z^d$ will contain a subset with at least $M'$ points separated from each other by a distance of at least $2n+1$. In this way, the trials will be independent.

\item[(3bis)] Choice of $L,T$: We want at least $M$ living points on the orthant so we want to choose $L,T$ such that $\P\left( N_+^f(L,T)>M\right)\geq 1-\epsilon$. To use Lemma~\ref{FKG2}, we have to control the number of living points in all the lateral faces. Using Lemma~\ref{etape2} we derive that for some $j$, one has
\[\P\left(N^{f_n} (L_j,T_j)> Md2^d\right)> 1-\epsilon.\]
Letting $L=L_j,T=T_j$ for that choice of $j$, we apply Lemmas~\ref{FKG1} and~\ref{FKG2} and get that
\begin{align*}
\P\left(|\leftidx{_{L}}{A}{_{T}^{f_n}}\cap \left[\right. 0,L \left.\right)^d|>N\right)&\geq 1-\epsilon^{2^{-d}}~~\text{  and }\\
\P\left(N_+^{f_n} (L,T)> M\right)&\geq 1-\epsilon^{2^{-d}/d}.
\end{align*}
\item Using the independence of Poisson processes on disjoint space-time regions, we obtain:
\begin{align*}
 \P&\left(\exists x\in [0,L)^d \leftidx{_{L+2n}}{\xi}{_{T+1}^{f_n}} \supset T_x\circ f_n\right)\\
 &\geq\P\left(|\leftidx{_{L}}{A}{_{T}^{f_n}}\cap [0,L[^d|>N\text{ and one of these points spawns }f_n\right)\\
 &\geq (1-\epsilon^{2^{-d}})(1-\epsilon),
\end{align*}
and
\begin{align*}
   \P&\left(\exists x\in \{L+n\}\times[0,L)^{d-1} \text{ and } t\in[0,T) \text{ such that }\leftidx{_{L+2n}}{\xi}{_{t+1}^{f_n}} \supset T_x\circ f_n \right)\\
   &\geq\P\left(N_+^{f_n} (L,T)> M \text{ and one of these points spawns }f_n\right)\\
 &\geq (1-\epsilon^{2^{-d}/d})(1-\epsilon).
\end{align*}
\end{enumerate}
Choosing initially $\epsilon^{2^d}$ instead of $\epsilon$, we obtain the announced result.
\end{proof}

\subsection{Finite space-time condition implies supercritical}
We have worked to construct ``good events'' on large finite space-time boxes. But why are they called good? Because these events satisfy the reverse proposition: if the events occur with high probability, the process is supercritical. So the supercriticality can be characterized by looking at the Poisson processes in a finite space-time box, without the need to know the entire process. In this part, we suppose that the finite space time conditions \eqref{C1} and \eqref{C2} hold and we prove the supercriticality of the process by embedding a supercritical percolation process in the CPA. Before starting this construction, we obtain other equivalent conditions. 

First, we combine the conditions \eqref{C1} and \eqref{C2} to obtain a block event guaranteeing the existence of a well-oriented open path. 

\begin{lemma}
\label{lemC}
Let $(f_n)_n$ be an increasing sequence of acceptable functions from $\Z^d$ to $\N$ which converges to $f:\Z^d\to\N$ not identically null and periodic. Assume that for every $\epsilon>0$ there exist $n,L,T\in\N^*$ such that conditions \eqref{C1} and \eqref{C2} are satisfied. Then 
  \[
  \P\left(\begin{array}{c}
           \exists x\in [L+n,2L+n]\times[0,2L)^{d-1} \text{ and } t\in[T,2T)\\
              \label{C}\tag{C}
  \text{ such that }\leftidx{_{2L+3n}}{\xi}{_{t}^{f_n}} \ge T_x\circ f_n
  \end{array}
  \right)> (1-\epsilon).
  \]
\end{lemma}
\begin{proof}
The idea is to use condition \eqref{C2} to spawn (with high probability) $f_n$ shifted in space but not necessarily in time, and use condition \eqref{C1} with the Markov property on the restarted process to spawn $f_n$ shifted in time but not necessarily in space.

For $n,L,T\in\N^*$, we define $E_1^{n,L,T}$ and $E_2^{n,L,T}$ as the events appearing in conditions \eqref{C1} and \eqref{C2}: 
\begin{align*}
E_1^{n,L,T}&=\left\{\exists x\in [0,L)^d \text{ such that } \leftidx{_{L+2n}}{\xi}{_{T}^{f_n}} \ge T_x\circ f_n \right\}\\
E_2^{n,L,T}&=\left\{\exists x\in \{L+n\}\times[0,L)^{d-1} \text{ and } t\in[0,T)\text{ such that }\leftidx{_{L+2n}}{\xi}{_{t}^{f_n}} \ge T_x\circ f_n\right\}.
\end{align*}
If $E_2^{n,L,T}$ occurs, then let $X_2$ be the first point (according to some fixed order in space) and $T_2$ be the first time such that $\leftidx{_{L+2n}}{\xi}{_{T_2}^{f_n}} \ge T_{X_2}\circ f_n$. Let $\epsilon>0$ and $n,L,T$ given by Theorem~\ref{C1etC2} such that $\P(E_1^{n,L,T})>1-\frac{\epsilon}{2}$ and $\P(E_2^{n,L,T})>1-\frac{\epsilon}{2}$.
For $n,L,T\in\N^*$, we define $E^{n,L,T}$ the event appearing in condition \eqref{C} as
\[
E^{n,L,T}=\left\{\begin{array}{c} \exists x\in [L+n,2L+n]\times[0,2L)^{d-1} \text{ and } t\in[T,2T)\\
\text{ such that }\leftidx{_{2L+3n}}{\xi}{_{t}^{f_n}} \ge T_x\circ f_n\end{array}\right\}.
\]
We have
\begin{align*}
\P(E^{n,L,T})&\geq \P\left(\left(\xi_s^{f_{n}}\right)_{s\geq 0} \in E_2^{n,L,T} \text{ and } \left(\xi_{T_2+s}^{f_{n}}\right)_{s\geq 0} \in E_1^{n,L,T}\right)\\
\label{un}\tag{*}
&\geq \P\left(\left(\xi_s^{f_n}\right)_{s\geq 0} \in E_2^{n,L,T}\right) \P\left(\Big(\xi_{s}^{g_n}\Big)_{s\geq 0} \in E_1^{n,L,T}\right)\\
&\geq \left(1-\frac{\epsilon}{2}\right)^2.
\end{align*}
In \eqref{un}, $g_n$ is the random state of the CPA at the stopping time $T_2$ where all the particles in the box $X_2+[-n,n]^d$ are alive (conditionally to $E_2$). The inequality \eqref{un} follows from the strong Markov property applied at the stopping time $T_2$, from the monotonicity and from the spatial invariance of the model. 
\end{proof}

The purpose is now to use the previous lemma in order to construct an active path with \textit{good direction}, and with the same \textit{uncertainty} at the start and finish.
To establish Theorem~\ref{FST}, we now take for every $n\ge 0$, $f_n=\1_{[-n,n]^d}$.
\begin{lemma} 
\label{Cbloc}
Suppose that  $\eqref{C}$ is satisfied. Then for every $\epsilon>0$, there exist $n,a,b\in\N^*$ with $n<a$ such that if $(x,s)\in[-a,a]^d\times[0,b]$ then
\begin{align*}
\tag{B1}
\P\left( \begin{array}{c}
          \exists(y,t)\in[a,3a]\times[-a,a]^{d-1}\times[5b,6b]\text{ and open paths staying}\\
           \text{in }[-5a,5a]^d\times[0,6b]\text{ and going from }(x,s)+[-n,n]^d\times\{0\}\\
           \text{ to every point in } (y,t)+[-n,n]^d\times\{0\}
         \end{array}
\right)> 1-\epsilon.
\end{align*}
\end{lemma}
\begin{proof} We refer the reader to the proof of Proposition 2.22 in~\cite{lig99} or directly to the initial proof in~\cite{bezgrim}. The main idea is to use repeatedly (and in a somewhat tricky way) the previous lemma in order to guide our path. The essential arguments are the strong Markov property satisfied by $\left(\xi_t\right)$ and the fact that we have forced the open path to stay in large boxes so as to ensure that the Poisson processes involved in the procedure use disjoint time intervals. \end{proof}
\begin{figure}[!ht]
\begin{center}
\begin{tikzpicture}[scale=0.8]
 \draw [->] (-8,0)--(8,0);
 \draw (8,0) node[below right] {$x$};
 \draw [->,gray] (0,-1)--(0,7);
 \draw (0,7) node[left] {$t$};
  \draw (0,0) node[below left] {$0$};
\draw (-1.5,0) node[below] {$-a$};
\draw (1.5,0) node[below] {$+a$};
\draw (0,1) node{$-$} node[above left] {$b$};
\draw (4.5,0) node[below] {$3a$};
\draw (7.5,0) node[below] {$5a$};
\draw (0,6) node{$-$} node[above left] {$6b$};
\draw (0,5) node{$-$} node[left] {$5b$};
\draw[dashed] (1.5,0)--(1.5,6);
\draw[dashed] (4.5,0)--(4.5,6);
\fill[gray,opacity=0.5] (-1.5,0)--(-1.5,1)--(1.5,1)--(1.5,0)--cycle;
\fill[gray,opacity=0.2] (-7.5,0)--(-7.5,6)--(7.5,6)--(7.5,0)--cycle;
\fill[gray,opacity=0.5] (1.5,5)--(1.5,6)--(4.5,6)--(4.5,5)--(1.5,5)--cycle;
\draw [ultra thick] (-0.2,0.7)--(1.1,0.7);
\draw (0.55,0.7) node[below] {$(x,s)$};
\draw [ultra thick] (1.8,5.1)--(3.1,5.1);
\draw (2.35,5.1) node[above] {$(y,t)$};
\draw[->,>=latex] (0.55,0.7) to[out=90,in=-135] (6.35,2.8);
\draw[->,>=latex] (6.35,2.8) to[out=45,in=-90] (2.35,5.1);
\end{tikzpicture}
\end{center}
\caption{The block event of Lemma~\ref{Cbloc} (condition \eqref{B1}}
\label{dessinbloc} 
\end{figure}
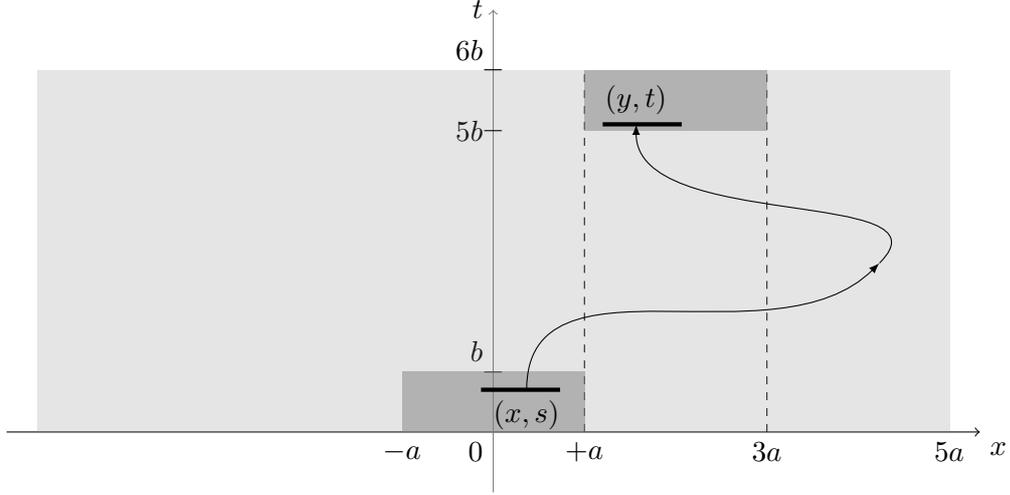

Using Lemmas~\ref{C1etC2}, \ref{lemC} and~\ref{Cbloc}, we obtain a part of Theorem~\ref{FST}. This will be the elementary brick for building a macroscopic oriented percolation process.

If the event in Figure~\ref{dessinbloc} occurs, we will say that the edge from the first gray cube to the second gray cube is open. But this event depends on the initial point $x$ in the first cube, and the next step of the construction depends on the location of the new departure, that is the center of the second cube. For this reason, we have to perform a dynamical renormalization: we thus define an oriented percolation by induction.

Let $\mathcal{L}=\left\{(j,k)\in\Z\times\N\text{ such that }j+k\text{ even}\right\}$ be the sites of the macroscopic grid of our renormalization process and $\mathcal{E}=\left\{(j,k)\rightarrow (j',k+1)\text{ with }|j-j'|=1\right\}$ its edges.

\begin{theorem}\label{construction} Let $\epsilon>0$. Suppose that the condition~\eqref{B1} is satisfied and let $n,a,b$ given by Theorem~\ref{FST}.
Then, there exists a random field $(W^k_e)_{e\in\mathcal{E},k\ge 1}$ taking values in $\{0,1\}$ and a filtration $(\mathcal{G}_k)_{k\ge 1}$ such that: 
\begin{itemize}
\item  $\forall k\ge 1,\forall e \in \mathcal{E}\quad W^k_e\in\mathcal{G}_k$;
\item $\forall k \ge 0,\forall e \in \mathcal{E}\quad \P[W^{k+1}_e=1|\mathcal{G}_k\vee \sigma(W^{k+1}_f, d(e,e')\ge 2)]\ge 1-\epsilon$;
\item if there exists an open path from $(0,0)$ to $(j,k)$ in $W$, then $\xi_{30bk}^{f_n}\not\equiv 0$;
\end{itemize}
where $\sigma(W^{k+1}_f, \; d(e,e')\ge 2)$ is the $\sigma$-field generated by the random variables $W^{k+1}_{e'}$, with $d(e,e')\ge 2$; and $f_n=\1_{[-n,n]^d}$.
\end{theorem}

\begin{proof}
We consider a contact process with aging $(\xi_t)$ with initial configuration $f_n=\1_{[-n,n]^d}$.

An open edge in the macroscopic grid will match with the existence of a good open path in the microscopic grid in such a way that the survival of a percolation in $\mathcal{L}$ leads to the survival of the corresponding CPA.
An open path between two fixed points occurs with small probability; consequently, the vertices of the macroscopic grid will not be points but zones of the microscopic grid.

For $(j,k)\in\mathcal{L}$, the event block of Lemma~\ref{Cbloc} leads us to introduce
\[\mathcal{S}_{j,k}=[a(2j-1),a(2j+1)]\times [-a,a]^{d-1}\times[b5k,b(5k+1)],\] the shaded area corresponding to the starting and ending zones of paths, and  \[\mathcal{A}_{j,k}=[a(2j-5),a(2j+5)]\times [-a,a]^{d-1}\times[b5k,b5(k+1)],\] the area where we look for paths from $\mathcal{S}_{j,k}$ to $\mathcal{S}_{j+1,k+1}$. 

 In order to control the dependence between the explored zones we will rather consider a diagonal block composed by $6$ elementary blocks. Thus, our new block event will be the following: if we set for all $(j,k)\in\mathcal{L}$
 \[ \mathcal{S}^6_{j,k}=\mathcal{S}_{6j,6k},~ \mathcal{A}_{j,k}^6=\bigcup_{i=0}^5 \mathcal{A}_{6j+i,6k+i}\text{ and } \tilde{\mathcal{A}}_{j,k}^6=\bigcup_{i=0}^5 \mathcal{A}_{6j-i,6k+i},\]
then, for $(x,s)\in \mathcal{S}^6_{j,k}$ one has
\begin{align}
\label{finalbloc}\tag{B2}
\P\left( \begin{array}{c}
          \exists(y,t)\in\mathcal{S}^6_{j+1,k+1}\text{ and paths staying in }\mathcal{A}_{j,k}^6\text{ and going from}\\
          (x,s)+[-n,n]^d\times\{0\}\text{ to every point in } (y,t)+[-n,n]^d\times\{0\}
         \end{array}
\right)> 1-\epsilon.
\end{align}
and (event in the other direction)
\begin{align*}
\P\left( \begin{array}{c}
          \exists(y,t)\in\mathcal{S}^6_{j-1,k+1}\text{ and paths staying in }\tilde{\mathcal{A}}^6_{j,k}\text{ and going from}\\
           (x,s)+[-n,n]^d\times\{0\}\text{ to every point in } (y,t)+[-n,n]^d\times\{0\}
         \end{array}
\right)> 1-\epsilon.
\end{align*}

Now we want to define random variables $(W^k_e)_{k\geq1, e\in \mathcal{E}}$ encoding states (open or close) of oriented edges between macroscopic sites. To do so, we must keep track of points  $Y_{j,k}\in\Z^d\times\R^+$ in the microscopic grid such that if there exists an open path in $(W^k_e)_{k\geq1, e\in \mathcal{E}}$ from the origin to the site $(j,k)$, then we have an open path from the origin of microscopic grid to $Y_{j,k}$ and we start from $Y_{j,k}$ to keep looking for an infinite path.

As the macroscopic percolation lives in $\Z\times \N$, we will denote an edge $e$ by its extremity $j$ and its direction $+$ or $-$.
\begin{figure}[!ht]
\begin{center}
\begin{tikzpicture}[scale=0.8]
 \draw [->] (-7.5,5)--(-6.3,5);
 \draw (-6.3,5) node[below] {$\mathbb{Z}^d$};
 \draw [->] (-7.3,4.8)--(-7.3,5.8);
 \draw (-7.3,5.8) node[right] {$\mathbb{R}^+$};
\fill[gray,opacity=0.5] (-1,0)--(-1,1)--(1,1)--(1,0)--cycle;
\fill[gray,opacity=0.5] (5,0)--(5,1)--(7,1)--(7,0)--cycle;
\fill[gray,opacity=0.5] (-7,0)--(-7,1)--(-5,1)--(-5,0)--cycle;
\fill[gray,opacity=0.2] (-7.5,0)--(-7.5,6)--(7.5,6)--(7.5,0)--cycle;
\fill[gray,opacity=0.5] (2,5)--(2,6)--(4,6)--(4,5)--(2,5)--cycle;
\fill[gray,opacity=0.5] (-2,5)--(-2,6)--(-4,6)--(-4,5)--(-2,5)--cycle;
\draw [ultra thick] (0.1,0.7)--(0.8,0.7) node[midway] {$\bullet$} ;
\draw (-2,2) to[bend left] (0.44,0.75);
\draw (-2,2) node[left] {$Y_{j,k}$};
\draw [ultra thick] (6.3,0.9)--(7,0.9) ;
\draw[->,>=latex] (1.95,0) to[out=45,in=-90] (0.65,0.7);
\draw[->,>=latex] (0.65,0.7) to[out=90,in=-135] (6.35,2.8);
\draw[->,>=latex] (6.35,2.8) to[out=45,in=-90] (2.65,5.5);
\draw [ultra thick] (2.1,5.5)--(2.8,5.5) node[midway] {$\bullet$} ;
\draw[] (1.4,4.4) to[bend right] (2.44,5.45);
\draw (1.4,4.4) node[left] {$Y_{j+1,k+1}$};
\draw[->,>=latex,dashed] (6.65,0.9) to[out=60,in=-45] (3.35,3.1);
\draw[->,>=latex,dashed] (3.35,3.1) to[out=120,in=-90] (3.6,5.1) ;
\draw [ultra thick, dashed] (3.25,5.1)--(3.95,5.1);
\draw (-0.5,0.5) node {\bf \textcolor{red}{$\mathcal{S}^6_{j,k}$}};
\draw (6,0.5) node {\bf \textcolor{red}{$\mathcal{S}^6_{j+2,k}$}};
\draw (3.8,5.7) node {\bf \textcolor{red}{$\mathcal{S}^6_{j+1,k+1}$}};
\draw (-3,5.5) node {\bf \textcolor{red}{$\mathcal{S}^6_{j-1,k+1}$}};
\draw[red] (-6,0.5) node {\bf $\mathcal{S}^6_{j-2,k}$};
\end{tikzpicture}
\begin{tikzpicture}[scale=0.8]
\draw [->] (7.2,5)--(8.4,5);
 \draw (8.4,5) node[below] {$\mathbb{Z}$};
 \draw [->] (7.4,4.8)--(7.4,5.8);
 \draw (7.4,5.8) node[right] {$\mathbb{N}$};
\draw (9,0) node {$\bullet$};
\draw (15,0) node {$\bullet$};
\draw (21,0) node {$\bullet$};
\draw (18,5) node {$\bullet$};
\draw (12,5) node {$\bullet$};
\draw (9,0) node[below] {$(j-2,k)$};
\draw (15,0) node[below] {$(j,k)$};
\draw (21,0) node[below] {$(j+2,k)$};
\draw (18,5) node[above] {$(j+1,k+1)$};
\draw (12,5) node[above] {$(j-1,k+1)$};
\draw[->,>=latex] (15.1,0.1) -- (17.9,4.9) node[midway,above,sloped] {$W^{k+1}_{j,+}$};
\draw[->,>=latex,dashed] (20.9,0.1) -- (18.1,4.9) node[midway,above,sloped] {$W^{k+1}_{j+2,-}$};
\draw[->,>=latex,gray,opacity=0.5] (9.1,0.1) to (11.9,4.9);
\draw[->,>=latex,gray,opacity=0.5] (14.9,0.1) to (12.1,4.9);
\end{tikzpicture}
\end{center}
\caption{Percolation background process}
\label{backproc}
\end{figure}
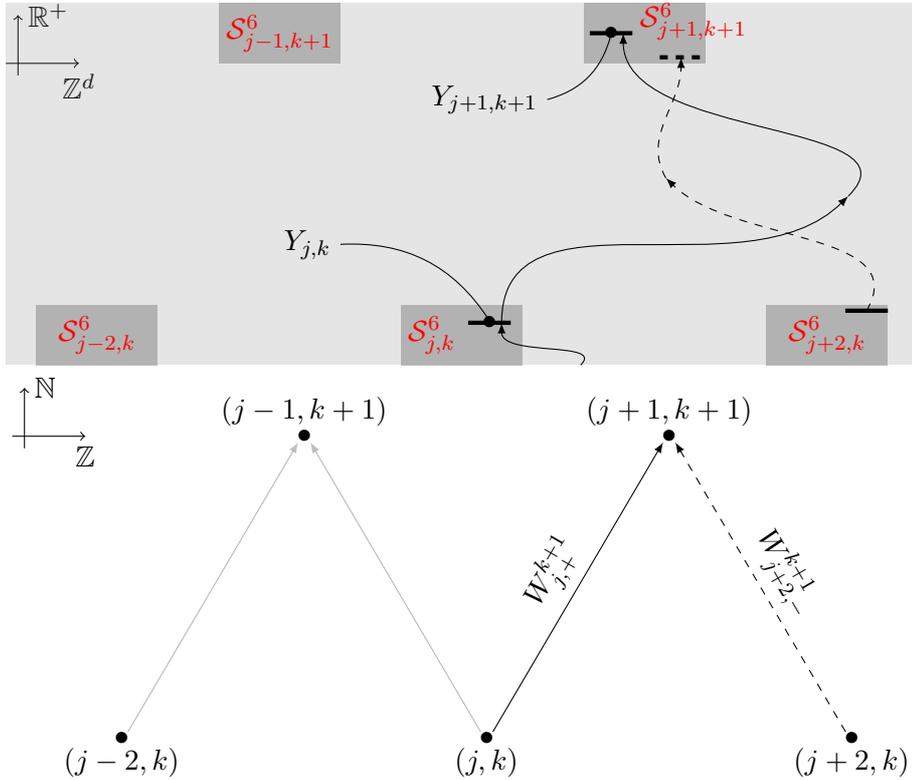

We want to define $Z_k(j)=(W^k_{j,+},W^k_{j+2,-},Y_{j,k})$ for $(j,k)\in\mathcal{L}$, with $W^k_{j,u}$ taking values in $\{0,1\}$ and $Y_{j,k}$ taking values in  $\Z^d\times\R^+\cup \{\dagger\}$ where $\dagger$ represents a cemetery point. We will need $(B^k_{j,u})_{k\geq1, j\in \Z,u=\pm}$ independent Bernoulli random variables of parameter $1-\epsilon$ (and independent of the rest of the construction).

We define the process by induction on $k\in\N$ (the construction is illustrated by Figure~\ref{backproc}):
\begin{itemize}
 \item $k=0$ : let $Y_{0,0}=(0,0)$ and $Y_{i,0}=\dagger$ for $i\neq 0$. This expresses the fact that every point in $[ -n,n]^d$ is alive at time $0$ (and only them).
  \item $k\Rightarrow k+1$ : suppose that the process $(Z_m(j))_{j,m}$ is defined for $j\in\Z$ and $0\leq m\leq k$, and let us define $(Z_{k+1}(j))_{j\in\Z}$. Let $j\in\Z$.
  \begin{itemize}
  \item If $Y_{j,k}=Y_{j+2,k}=\dagger$, then $W^{k+1}_{j,+}=B^{k+1}_{j,+}$, $W^{k+1}_{j+2,-}=B^{k+1}_{j+2,-}$ and $Y_{j+1,k+1}=\dagger$.
   \item Suppose $Y_{j,k}\neq \dagger$; if there exists $X\in \mathcal{S}^6_{j+1,k+1}$ such that there are paths from $Y_{j,k}+[-n,n]^d$ to every point in $X+[-n,n]^d\times\{0\}$ staying in $\mathcal{A}_{j,k}^6$, then $W^{k+1}_{j,+}=1$. $X$ is not necessarily unique so we choose $X_{j+1,k+1}$ as the first one in time and according to some fixed order in space. Otherwise, $W^{k+1}_{j,+}=0$ and $X_{j+1,k+1}=0^{d+1}$.
 \item  Suppose $Y_{j+2,k}\neq \dagger$; if there exists $X\in \mathcal{S}^6_{j+1,k+1}$ such that there are paths from $Y_{j+2,k}+[-n,n]^d$ to every point in $X+[-n,n]^d\times\{0\}$ staying in $\tilde{\mathcal{A}}_{j+2,k}^6$, then $W^{k+1}_{j+2,-}=1$. $X$ is not necessarily unique so we choose $\tilde{X}_{j+1,k+1}$ as the first one in time and according to some fixed order in space. Otherwise, $W^{k+1}_{j+2,-}=0$ and $\tilde{X}_{j+1,k+1}=0^{d+1}$.
 \item Finally \[ Y_{j+1,k+1}=\left\{\begin{array}{cc}
                                       \dagger\phantom{aaaaaaaaa}               & \text{ if }X_{j+1,k+1}=\tilde{X}_{j+1,k+1}=0^{d+1}\\
                                       \max\{ X_{j+1,k+1},\tilde{X}_{j+1,k+1}\} & \text{ else,}
                                      \end{array}\right.
\]
where the maximum is taken according to the time coordinate. The first case reflects the fact that there is no open path from the origin to the zone $\mathcal{S}^6_{j+1,k+1}$; the second case reflects the fact that there is at least one such path and we choose the farthest reached point to restart in the next step of construction.
 \end{itemize}
\end{itemize}
As intended, if there exists an infinite open path in the percolation diagram defined by the $(W^{k}_e)_{k\geq 1,e\in\mathcal{E}}$ then $(\xi_t^{f_n})$ survives in $\Z\times[-5a,5a]^d\times\R^+$; it is easy to see that the $(B^k_e)_{k,e}$ do not impact. Now we prove that this event has positive probability. 
Let $\mathcal{G}_k=\sigma\big(Z_l(i),i\in\Z,0\leq l \leq k\big)$; $\mathcal{G}_k$ contains the discovered area to construct the $k$ first steps of the percolation, so $\mathcal{F}_{30kb}\subset\mathcal{G}_k\subset \mathcal{F}_{(30k+1)b}$. For $(j,k)\in\mathcal{L}$, one has
\begin{align*}
&\P\left( W^{k+1}_{j,+}=1|\mathcal{G}_k\right)=\1_{\{Y_{j,k}=\dagger\}}\P(B^{k+1}_{j,+}=1)\\
&\phantom{\P}+\1_{\{Y_{j,k}\neq \dagger\}}\P\left( \begin{array}{c}
                                                   \exists X\in\mathcal{S}^6_{j+1,k+1}\text{ and paths staying in }\mathcal{A}_{j,k}^6\text{ and going}\\
                                                    \text{ from }Y_{j,k}+[-n,n]^d\times\{0\}\text{ to every point in } X+[-n,n]^d\times\{0\}
                                                    \end{array}
\right).
\end{align*}
So, using the block events \eqref{finalbloc} it holds that 
\[\P(W^{k+1}_{j,+}=1|\mathcal{G}_k)>(1-\epsilon).\]
We have that, conditioned on $\mathcal{G}_k$, the collection of variables $\{W^{k}_e,k\ge 1,e\in\mathcal{E}\}$ is locally dependent because the constructed paths stay in the $(\mathcal{A}_{k,i}^6)_{i\in\Z}$ or the $(\tilde{\mathcal{A}}_{k,i}^6)_{i\in\Z}$. This achieves the proof of Theorem~\ref{construction}.
\end{proof}
Now that the percolation process is constructed, we conclude this section by the converse in Theorem~\ref{FST}:
\begin{corollary}
Suppose that the condition \eqref{B1} is satisfied. Then $(\xi_t)$ is supercritical.
\end{corollary}
\begin{proof}
In the previous construction, we have coupled $(\xi_t)$ with $(W^{k}_e)$. The process $(W^{k}_e)$ is in the class $\mathcal{C}_d(M,q)$ (with $d=1,M=2,q=1-\epsilon$) introduced by Garet and Marchand in~\cite{DOP}.
We conclude by using Lemma 2.4 in the latter reference to couple our process with an independent oriented percolation process of parameter $g(1-\epsilon)$, where $g$ is a function from $[0,1]$ to $[0,1]$ with $\lim_{x\rightarrow 1} g(x)=1$. If we choose $1-\epsilon$ large enough, the independent percolation underlying survives, so our process $(\xi_t)$ survives as well. \end{proof}

\section{Consequences}\label{Scontrols}
\subsection{A \textit{dies out}}
A first easy consequence of the construction of Section~\ref{Sbezgrim} is the following result.
\begin{theorem} The two-stage contact process of Krone (\cite{krone}) dies out, that is for fixed $\gamma$, $\P_{\lambda_c(\gamma)}(\forall t>0,~\xi_t^{0(1)}\not\equiv0)=0$. \end{theorem}
\begin{proof} By Theorem~\ref{FST}, the supercriticality is characterized by a finite space-time condition which depends continuously on $\lambda$ and $\gamma$. The supercritical region is an open subset of $\R^{2}$ and for a critical value $\lambda_c(\gamma)$ the process can not survive.
\end{proof}

\subsection{At most linear growth}
We start this section by an independent but fundamental result. We show that the growth of the CPA is at most linear thanks to a comparison with Richardson model.

The Richardson model $(\eta_t)_{t\ge0}$ of parameter $\lambda$ is the continuous-time Markov process taking its values in $\mathcal{P}(\Z^d)$, such that:
\begin{itemize} 
 \item a dead site $z$ becomes alive at rate $\lambda\sum_{\|z-z'\|_1=1} \eta_t(z')$,
 \item a living site never dies.
\end{itemize}
Thanks to the graphical construction, we can build a coupling between the contact process with aging of parameters $(\Lambda,\gamma)$ and the Richardson model with parameter $\lambda_{\infty}$ such that at any time $t$, the set of living points in the contact process with aging is contained in the set of living points in the Richardson model. For this, the Richardson process does not see the deaths and the maturations and can use all the arrows. This leads to the following result.

\begin{lemma}\label{richardson}
Let $H_t^f=\cup_{s\le t} A_s^f$.
 There exist $A,B,M>0$ such that for every $f:\Z^d\to\N$ and every $t>0$
 \begin{align*}
\P(H_t^f\nsubseteq B_{Mt})\leq \P(\eta_t^{\supp f}\nsubseteq B_{Mt})\leq A\exp(-Bt).
 \end{align*}
\end{lemma}
The Richardson model is described in~\cite{richardson} and the proof of the second inequality of the lemma is detailed in~\cite{GMPCRE}.

\subsection{At least linear growth}
From now, we fix $\Lambda,\gamma$ such that $(\xi_t)$ survives with positive probability, so the condition~\eqref{B1} is satisfied. We fix an initial configuration $f:\Z^d\to\N$ with finite support. Let $\epsilon$ large enough, let $n,a,b$ be the parameters given by Theorem~\ref{FST} and denote by $W$ the percolation process (surviving) constructed in Theorem~\ref{construction} of Section~\ref{Sbezgrim}. We set, for $n\in\N$ and $x,y\in\Z$
\begin{align*}
\eta_n^W&=\{y\in\Z,~(0,0)\rightarrow (y,n)\text{ in }W\};\\
\tau^{W}&=\inf\{n\in\N,~ \eta_n^W=\emptyset\}, \text{ the extinction time of }W;\\
H_n^W&= \bigcup_{0\leq k\leq n} \eta_{k}, \text{ the set of all points touched before time }n;\\
\gamma'(\theta,x,y)&=\inf\left\{n\in\N: \forall k\ge n~\Card \{i\in\{0,\ldots,k\}, (x,0)\to(y,i)\}\ge\theta k\right\}.
\end{align*}

We recall here the properties of the percolation $(W_e^k)$ we built in Theorem~\ref{construction}. The proofs can be found in~\cite{DOP}. We will use our coupling to prove the sought exponential decays (Theorems~\ref{petitsamas} and~\ref{au-lineaire}).

\begin{corollary}\label{DOPexpo} Let $\alpha>0$. There exists $\beta>0$ such that \[\E[\1_{\{\tau^{W}<+\infty\}} \exp(\beta\tau^{W})]\leq \alpha.\]
\end{corollary}

\begin{corollary}\label{DOPcroissance}
There exist $A,B,C>0$ such that for every $L,n>0$, one has
\[\P(\tau^{W}=+\infty,[-L,L]\not\subset H^w_{CL+n})\leq Ae^{-Bn}.\]
\end{corollary}

\begin{corollary}
 \label{DOPretouche}
There exist positive constants $A',B',\theta,\beta$  such that for all $x,y\in\Z$, $n\ge 0$, one has \[\P(\beta |x-y|+n<\gamma'(\theta,x,y)<+\infty)\le A'e^{-B'n}.\]
\end{corollary}

Note that the CPA can survive without the background percolation process does. To be able to adapt the exponential controls from the percolation to the CPA, we have to build a \textit{restart argument} (see for instance Theorem 2.5 of~\cite{DOP} and Theorem 2.30 of~\cite{lig99} for similar constructions).

We recall that for all $t\in\R^+$, $A_t^f=\supp\xi_t^f$.

\begin{lemma}[Restart argument]
 There exist random variables $\sigma$, $Y$ such that
\begin{itemize}
\item $\sigma$ takes values in $\R^{+}$ and $Y$ takes values in $\Z^d\cup\{\dagger\}$.
 \item On the event $\{\tau^f<\infty\}$, one has $\sigma>\tau^f$ and $Y=\dagger$.
 \item On the event $\{\tau^f=\infty\}$, it holds that $Y$ lives in the slab $\Z\times[-a,a]^{d-1}$. Moreover, $Y+[-n,n]^d\subset A_\sigma^f$ and a macroscopic percolation $(W_e^{n})\circ\theta_\sigma\circ T_Y$ which almost surely survives starts from $Y+[-n,n]^d$ at time $\sigma$.
\end{itemize}
\end{lemma}

\begin{proof}[Construction of $Y,\sigma$:]

The first step is to control the time when we have enough living points to start a coupling with a percolation process. Then, once the coupling is started, the percolation can die out. If the process dies out too, then we have a link between the extinction times; otherwise, we restart the procedure. The quantity of interest is the sum of these two times (starting+extinction) over each iteration of the procedure.
 
To construct the starting time, we use the simple following property: there exists $\alpha>0$ such that \begin{multline*}
\P\left(\exists x\in\Z\times[-a,a]^{d-1}, A_1^f\supset x+[-n,n]^d\right)\\\geq\P\left(\exists x\in\Z\times[-a,a]^{d-1}, A_1^{\delta_0}\supset x+[-n,n]^d\right)\ge \alpha.                                                                                             \end{multline*}
Procedure (illustrated by Figure~\ref{dessinrestartproc}):\\
We first look for a time when a whole cube $[-n,n]^d$ is occupied by living particles. Let 
 \[N_1=\inf\left\{m: A^f_{m+1}=\emptyset\text{ or }\exists x\in\Z\times[-a,a]^{d-1} \text{ such that } A_{m+1}^f\supset x+[-n,n]^d\right\}.\]
 For every $m\in\N$, we have
 \begin{multline*}
  \P(N_1=m)\\
  =\P\left(\{A_{m+1}^f=\emptyset \text{ or }\exists x\in\Z\times[-a,a]^{d-1} \text{ such that } A_{m+1}^f\supset x+[-n,n]^d\}\cap\{N_1\geq m\}\right)\\
  = \P\left(A_{m+1}^f=\emptyset \text{ or }\exists x\in\Z\times[-a,a]^{d-1} \text{ such that } A_{m+1}^f\supset x+[-n,n]^d| N_1\geq m\right)\\
  \times\P\left(N_1\geq m\right)\\
  \geq \P\left(\exists x\in\Z\times[-a,a]^{d-1} \text{ such that } A_{1}^{\delta_0}\supset x+[-n,n]^d\right)\P\left(N_1\geq m\right)\\
  \geq \alpha~ \P\left(N_1\geq m\right).
 \end{multline*}

Thus $N_1$ is a sub-geometric random variable such that at time $N_1$, one has either
\[\{A_{N_1+1}^f=\emptyset\}\text{ or }\{\exists x\in\Z\times[-a,a]^{d-1}\text{ such that }A_{N_1+1}^f\supset x+[-n,n]^d\}.\] 
\begin{itemize}
 \item If $A_{N_1+1}^f=\emptyset$, then we take $M_1=0$ and $Y=\dagger$.
 \item If there exists $x\in\Z\times[-a,a]^{d-1}$ such that $A_{N_1+1}^f\supset x+[-n,n]^f$, then we denote by $X_1$ the first such $x$ (according to some fixed order). We couple the process starting at time $N_1+1$ with initial set $\xi_{N_1+1}^f$ with a percolation process starting from this cube, as previously thanks to \eqref{B1}. If the percolation survives then the CPA survives too and we stop the procedure taking $M_1=0$ and $Y=X_1$. If the percolation does not survive, then let $M_1$ be its extinction time. We look at the CPA at time $U_1=N_1+1+b30M_1$ that is the time in the microscopic grid corresponding to the extinction of the macroscopic process. At time $U_1$:
 \begin{itemize}
  \item if $A^f_{U_1}=\emptyset$, then we stop the procedure and set $Y=\dagger$,
  \item otherwise, we begin the whole procedure again with a CPA with initial set $\xi_{U_1}^f$ at time $U_1$. 
 \end{itemize}
\end{itemize}
We obtain sequences of independent random variables $(N_i)$ with the same distribution as $N_1$, independent random variables $(M_i)$ with the same distribution as $M_1$ conditioned to be finite, and a geometric random variable $L$ independent of the $N_i's$ and $M_i$'s which is the number of times the $(N,M)$ procedure is carried out. Let
\begin{align*}
 \sigma=\sum_{i=1}^L U_i\text{ with }U_i=N_i+1+b30M_i.
\end{align*}

By construction, at time $\sigma$, one has either $A_\sigma^f=\emptyset$ or $\tau^f=\infty$. It implies that
\begin{itemize}
 \item on the event $\{\tau^f<\infty\}$, $\sigma>\tau^f$ and $Y=\dagger$;
 \item on the event $\{\tau^f=\infty\}$, $Y$ lives in the slab $\Z\times[-a,a]^{d-1}$, $Y+[-n,n]^d\subset \xi_\sigma^f$ and a macroscopic percolation $(W_e^{n})\circ\theta_\sigma$ which almost surely survives starts from $Y+[-n,n]^d$ at time $\sigma$.\qedhere
\end{itemize}
\end{proof}
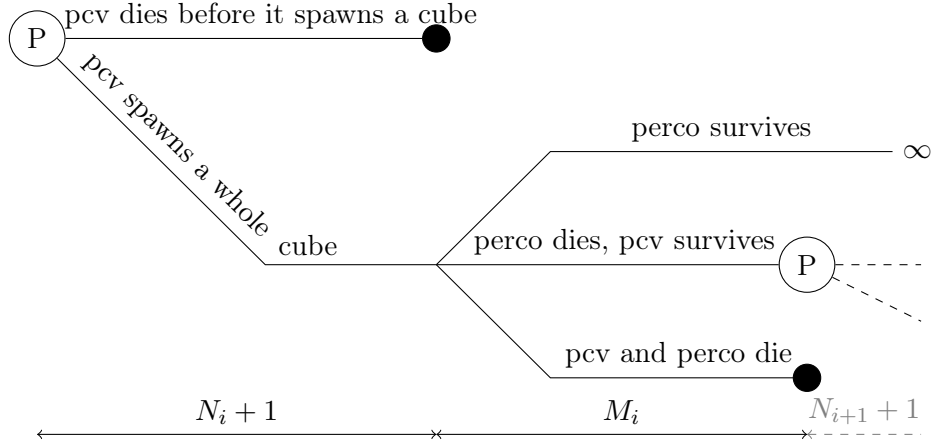
\begin{figure}
 \begin{center}
\begin{tikzpicture}[scale=0.75]
\node[draw,circle] (P) at (0,0){P};
\node[draw,circle] (R) at (13.5,-4){P};
\node[draw,circle,fill=black] (F) at (7,0){};
\node[draw,circle,fill=black] at (13.5,-6){};
\draw (P)--(F);
\draw (4.1,0) node[right,above]{pcv dies before it spawns a cube};
\draw (P)--(4,-4)node[midway,above,sloped]{pcv spawns a whole}--(7,-4)node[near start,above]{cube}--(R);
\draw (10.3,-4) node[right, above]{perco dies, pcv survives};
\draw (7,-4)--(9,-2)--(15,-2)node[right]{$\infty$} node[midway, above]{perco survives};
\draw (7,-4)--(9,-6)--(13.5,-6)node[midway,above]{pcv and perco die};
\draw[dashed](R)--++(2,0);
\draw[dashed](R)--++(2,-1);
\draw[<->](0,-7)--(7,-7) node[midway,above]{$N_i+1$};
\draw[<->](7,-7)--(13.5,-7) node[midway,above]{$M_i$};
\draw[<-,dashed,gray](13.5,-7)--(15.5,-7) node[midway,above,gray]{$N_{i+1}+1$};
\end{tikzpicture}
 \end{center}
 \caption{Restart procedure at step $i$}
 \label{dessinrestartproc}
\end{figure}

\begin{corollary}\label{controls}
Let $\sigma$ and $Y$ be the random variables constructed in the restart argument. Then there exist positive constants $A,B,A_2,B_2$ such that for every $t>0$:
\begin{align*}
 \P\left(\sigma>t\right)&\leq A\exp(-B t),\\
\P\big( \|Y\|>t\big)&\leq A_2\exp(-B_2 t),
\end{align*}
\end{corollary}
\begin{proof}
We start the process with some $f:\Z^d\to\N$. We use the previous construction to extend the well-known exponential control of extinction time of the percolation setting to our process. The random variables $L$ and $(N_i)$ have both exponentially decaying tail probabilities. The Corollary~\ref{DOPexpo}, say that the same is true for $M_i$. So, we take $\beta_1>0$ such that $\E[ e^ {\beta_1 L}]<\infty$, and then $\beta_2>0$ such that $\E[ e^{\beta_2 (N_1+1+30bM_1)}]\leq e^{\beta_1}$. Then 
\begin{align*}
\E\left[e^{\beta_2 \sigma}\right] &=\E\left[\E\left[e^{\beta_2\sigma}|L\right]\right]\\
&=\E\left[\E\left[e^{\beta_2\sum_{i=1}^L U_i}|L\right]\right]\\
&=\E\left[\E\left[e^{\beta_2(N_1+1+30bM_1)}\right]^L\right]\leq \E\left[e^{\beta_1 L}\right]<\infty.
\end{align*}
It follows that $\sigma$ has exponentially decaying tail probabilities, which is the first inequality.
Using the fact that the growth is at most linear and that $Y\in A_\sigma^f$, we can obtain the second inequality. For every $t\in\R^+$:
\begin{align*}
 \P\left(\|Y\|>t\right)&\leq  \P\left( \{\|Y\|>t \}\cap \{\sigma>ct\}\right)+ \P\left( \{\|Y\|>t\} \cap\{ \sigma\leq ct\}\right)\\
 &\leq \P\left( \sigma>ct\right)+\P\left( \exists~ Y\in H_{ct} \text{ with }\|Y\|>t\right)\\
 &\leq 2Ae^{-\frac{B}{M}t},
\end{align*}
taking $c=\frac{1}{M}$ where $M$ is the constant appearing in Lemma~\ref{richardson}.
\end{proof}
We deduce from Corollary~\ref{controls}:
\thmpetitsamas*
\begin{proof}
By construction, on the event $\{\tau^f<\infty\}$, $\sigma>\tau^f$. Therefore,
\[\P(t<\tau^f<\infty)\leq \P(\sigma>t)\]
and the expected result follows. 
\end{proof}
Now, we handle the second fundamental estimate.
\thmaumlineaire*
\begin{proof}The natural idea is to use the corresponding result about the percolation process and to adapt it to the CPA as we did it in the previous corollaries. But here Corollary~\ref{DOPcroissance} is not sufficient because when we hit a macroscopic point with $W$ we do not hit the whole associated zone in the microscopic grid, but only a part of it. To ensure that we hit every point in this zone we have to repeat the hit. For this reason, we will need Corollary~\ref{DOPretouche}. 

In order to use Corollary~\ref{DOPretouche}, we have to use the percolation construction again. However, we recall that our construction was made in a slab $\Z\times[-a,a]^{d-1}$. First, we explain how to deal with the case where $x$ is in the slab and then we will explain the general case. 

The restart argument gives the existence of random variables $Y\in\Z\times[-a,a]^{d-1}$, $\sigma\in \R^+$, as well as a macroscopic percolation $W$ surviving from $\overline{0}$ (which corresponds to $Y+\mathcal{S}_{0,0}^{6}$ at time $\sigma$), and living in a macroscopic grid $\mathcal{L}\circ\theta_\sigma \circ T_Y$.

Let $x\in\Z^d$ and $x_1$ be its first coordinate. Then there exists a unique pair $(j,r)\in\Z\times[0,2a[$ such that $x_1-Y_1=(2j-1)a+r$; we put $\overline{x}=j$ the macroscopic site nearest to $x$. 

Let $\overline{R}^a_0(\overline{x})=0$ and for $i\geq 1$, let
\[ \overline{R}^a_i(\overline{x})=\inf\{n> \overline{R}^a_{i-1} (\overline{x}): \text{ there exists a path from $0$ to $(\overline{x},n)$ in }W\}\]
be the $i$-th time when $W$ hits $\overline{x}$. In an analogous way, we set
\begin{align*}
 R^a_0(x)&=0 \\
 R^a_{i+1}(x)&=\inf\{t\ge R'_i:~ \exists y\in x+[-7a,7a]^d\text{ with }y\in A_t^f\}\text{ and }R'_i=R^a_i(x)+59b.
\end{align*}

$R_i^a(x)$ is the i-th time when we hit the box $ x+[-7a,7a]^d$ with the constraint that these times are  separated by at least $59b$.

By definition of $\gamma$, if we choose $k\geq \theta\gamma$, then the number of times when we hit $\overline{x}$ from $\overline{0}$ before time $\frac{k}{\theta}$ is at least $k$. So
\[
\overline{R}_k^a(\overline{x})\leq \frac{1}{\theta} \max\left( \gamma(\theta,\overline{0},\overline{x}),k\right).\]
We remark that if we hit $k$ times the point $\overline{x}$ in the macroscopic grid then we hit at least $k$ times the zone $x+[-7a,7a]^d$ in the microscopic grid, and these touches are separated from at least $59b$ (the height of two edges) so 
\[ R_k^a(x)\leq \sigma+\frac{31b}{\theta} \max\left( \gamma(\theta,\overline{0},\overline{x})\circ\theta_{\sigma},k\right).\]
Then, using Corollary~\ref{DOPretouche} on the points $\overline{0}$ and $\overline{x}$ we have
\begin{align}\label{DECgamma}\P\left(\gamma(\theta, \overline{0},\overline{x})\leq \beta |\overline{x}|+k\right)\geq 1-A'e^{-B'k}.
\end{align}
This leads us to
\begin{align*}
 \P\left(R_k^a(x)\leq k+\frac{31b}{\theta}\left(\frac{\beta}{2a}\left(\|x\|+k\right)+k\right)\right)&\geq \P\left(\{\gamma(\theta,\overline{0},\overline{x})\leq \frac{\beta}{2a} \left(\|x\|+k\right)+k\} \cap \{\sigma\leq k\}\right)\\
 &\geq \P\left(\{\gamma(\theta,\overline{0},\overline{x})\leq \beta |\overline{x}|+k\}\cap \{|\overline{0}|\leq k\} \cap \{\sigma\leq k\}\right)\\
 &\geq  \P\left(\gamma(\theta,\overline{0},\overline{x})\leq \beta |\overline{x}|+k\right)\\
 &\phantom{\geq}+\P\left(|\overline{0}|\leq k\right) +\P\left(\sigma\leq k\right)-2\\
 \label{etoile}\tag{*}
 &\geq 1-3Ae^{-Bk}
\end{align*}
applying Corollary~\ref{controls} and Equation~\eqref{DECgamma}. Taking $C_1=\frac{31b\beta}{2a\theta}$ and $C_2=\frac{31b\beta}{\theta}+2$ we obtain that
\[\P\left(R_k^a(x)\geq C_1\|x\|+C_2k\right)\leq Ae^{-Bk}.\]

Now that we have succeeded in controlling hitting times of points that are not far from $x$, we try to reach $x$ from these points.

Consider the event $B=\{\forall y\in x+[-7a,7a]^d,~ \exists t\le 58b \text{ such that } x\in A_t^{\delta_y}\}$
and also, for $n\ge 1$
\[A_n={\cap}_{i=0}^n\{R_i^a(x)<+\infty, \theta^{-{R_i^a(x)}}(B^c)\}.\]
Note that by construction, $\theta^{-R_i^a(x)}(B)$ is $\mathcal{F}_{R_{i+1}^a(x)}$-measurable, so, by the Markov property
\[\P(A_n|\mathcal{F}_{R_n^a(x)})=\1_{A_{n-1}\cap \{R_n^a(x)<+\infty\}}\P(B^c).\]
We can see that $\P(B)\geq c>0$ so for each $n\geq 1$, $\P(A_n)\leq (1-c)^n$. Finally, using \eqref{etoile} we obtain
\begin{align*}
\P\left(\tau^f=+\infty, t(x)\geq C_1\|x\|+C_2n\right)&\leq \P\left(\tau^f=+\infty, R_n^a(x)\geq C_1\|x\|+C_2n\right)\\
&\phantom{\leq}+\P\left(\tau^f=+\infty, R_n^a(x)\leq C_1\|x\|+C_2n\leq t(x)\right)\\
&\leq Ae^{-Bn}+(1-c)^{n}. 
\end{align*}
Finally, if $x$ is not in the slab, we repeat the previous construction through the path $(x_1,0,\ldots,0)\rightarrow(x_1,x_2,0,\ldots,0)\rightarrow\cdots\rightarrow(x_1,\ldots,x_d)$. The hitting time between two successive points satisfies the inequality above so the same result follows for the total hitting time.
\end{proof}

\section{Ergodicity}\label{Sergo}
The previous exponential controls are the ingredients required to show the ergodicity of the dynamical system  $(\Omega, \mathcal{F},\overline{\P},\tilde{\theta}_x)$. First we establish properties of $K(x)$ that will be very useful in Section~\ref{Sssadd}.
\begin{lemma}
 \label{Kgeom}
 \begin{itemize}
  \item For all $x\in\Z^d$ and $k\in\N$, $\P(K(x)>k)\le(1-\rho)^k$.
  \item Almost surely, for every $x\in\Z^d$, $(K(x)=k\text{ and }\tau^0=+\infty) \Longleftrightarrow(u_k(x)<+\infty\text{ and } v_k(x)=+\infty)$.
 \end{itemize}
\end{lemma}

\begin{proof}
We recall that $\rho$ is the constant given in \eqref{survival}. Let $x\in\Z^d$ and $k\in \N$. Applying the strong Markov property at time $u_{k+1}(x)$, we have:
\begin{align*}
\P(K(x)>k+1)& = \P(u_{k+2}(x)<+\infty) \\
& \le \P(u_{k+1}(x)<+\infty,v_{k+1}(x)<+\infty) \\
& \le \P(u_{k+1}(x)<+\infty,\tau^x \circ \theta_{u_{k+1}(x)}<+\infty) \\
 & \le \P(u_{k+1}(x)<+\infty)\P(\tau^x <+\infty)\\
& \le \P(u_{k+1}(x)<+\infty) (1-\rho)=\P(K(x)>k)(1-\rho),
\end{align*}
which proves the first point.
Suppose that $\{\tau^0=+\infty\}$. Let $x\in\Z^d$ and $k\in \N$. Applying the strong Markov property at time $v_k(x)$, we have
\begin{align*}
\P&(\tau^0=+\infty, v_k(x)<+\infty, u_{k+1}(x)=+\infty| \mathcal F_{v_k(x)})\\
&=  \1_{\{v_k(x)<+\infty\}}\P(\tau^{\centerdot}=+\infty, t^{\centerdot}(x)=+\infty) \circ \xi^0_{v_k(x)}.
\end{align*}
The last term is equal to zero due to the growth being at least linear (Theorem~\ref{au-lineaire}). We deduce the second point.
\end{proof}

It means that, on the event $\{\tau^{0}=\infty\}$, $K(x)$ is almost surely finite, $\sigma(x)$ is well defined and the procedure stops when we find a moment where $x$ has an infinite descent. Thanks to the previous equivalence and following the proofs of Lemmas 8 and 9 in~\cite{GMPCRE} we establish the \textit{law after the restart} and show the expected properties of invariance and independence. This is the content of the following result. 

\begin{corollary}
 \label{invariancePbarre}
Let $x$ and $y$ in $\Z^d$ and assume that $x \neq 0$.
\begin{itemize}
\item For $A$ in the $\sigma$-algebra  generated by $\sigma(x)$, and $B\in \mathcal F$ we have: $\Pbarre(A \cap (\tilde{\theta}_x)^{-1}(B))=\Pbarre(A)\Pbarre(B)$.
\item The probability measure $\Pbarre$ is invariant under the action of $\tilde \theta_x$.
\item Under $\Pbarre$, $\sigma(y)\circ\tilde{\theta}_x$ is independent from $\sigma(x)$ and it has the same law as $\sigma(y)$.
\item The random variables $(\sigma(x) \circ(\tilde\theta_{x})^j)_{j\ge0}$ are independent and identically distributed under $\Pbarre$.
\end{itemize}
\end{corollary}

The expected Theorem~\ref{ergo} is contained in the previous result.
\begin{remark}
In Lemma 10 of~\cite{GMPCRE}, the authors obtain the ergodicity of the system. We can do exactly the same thing in our setting but we only need the properties of invariance and independence for our asymptotic shape theorem.
\end{remark}

\section{Almost subadditivity and difference between $\sigma$ and $t$}\label{Sssadd}
Here, we want to bound the quantities $\sigma(x+y)-[\sigma(x)+\sigma(y)\circ\tilde{\theta}_x]$ and $\sigma(x)-t(x)$. For this, we need to go back to the definition of the essential hitting time $\sigma$ and to deal with two types of sum of random
variables: sum of $v_i-u_i$ on one hand, which are quite simple to control thanks to Theorem~\ref{petitsamas}, and sum of $u_{i+1}-v_i$ on the other hand. The quantity $u_{i+1}-v_i$ depends on the whole configuration of the process at time $v_i$ and thus is more difficult to handle. For this control, we will need all the results about the growth that we have previously established.  

\subsection{Easy case: $v_i-u_i$}

\begin{lemma}\label{vimoinsui}
There exist $A,B>0$ such that for all $x\in\Z^d$ and $t>0$ one has
\[\overline{\P} \left(\exists i<K(x)\text{ such that }v_i(x)-u_i(x)>t\right)\leq A\exp(-Bt).\]
\end{lemma}
\begin{proof}
Using \eqref{survival} and Markov property we have
\begin{align*}
\Pbarre(\exists i<K(x): v_i(x)-u_i(x) >t) & \le \frac{1}{\rho}\P\left(\bigcup_{i=1}^{+\infty}\{v_i(x)-u_i(x)>t\}\cap\{i<K(x)\}\right)\\ 
&\le \frac{1}{\rho}\sum_{i=1}^{+\infty}\P\left(\{t<v_i(x)-u_i(x)<+\infty\} \cap \{u_i(x)<\infty\}\right)\\
&\le \frac{1}{\rho}\sum_{i=1}^{+\infty}\P\left(\theta_{u_i(x)}^{-1}\left(\{t<\tau^x<\infty\}\right) \cap \{u_i(x)<\infty\}\right) \\
&\le \frac{1}{\rho}\P\left(t<\tau^x<\infty\right)\sum_{i=1}^{+\infty}\P\left( i-1<K(x)\right) \\
&\le \frac{1}{\rho^2}\P\left(t<\tau^x<\infty\right).
\end{align*}
We conclude using Theorem~\ref{petitsamas}.
\end{proof}

\subsection{Box of bad growth points}

To deal with the regeneration times $u_{i+1}(x)-v_i(x)$, the idea is to find a point close to $(x,u_i(x))$, descendant of $(0,0)$ and with infinite life time, which will quickly regenerate $x$. To this end, we look, around $x$, if there are bad growth points. For every $x\in\Z^d$, $L> 0$ and $t>0$, we introduce the event
\begin{align*}
E^y(x,t)  = &\{\omega_y[0, t/2]=0\} \cup\{H^y_{t}\not\subset y+B_{Mt}\} \cup \{ t/2<\tau^y<+\infty\} \\
\cup &\{\tau^y=+\infty, \inf\{s\ge2t: x\in A^y_{s}\}> \kappa t\}.
\end{align*}
traducing the fact that a point $y$ has bad growth with respect to the spatiotemporal point $(x,t)$ and we count the number of this points by
\[N_L(x,t)=\sum_{y\in x+B_{M t+2}} \int_0^L \1_{E^y(x,t)}\circ\theta_s \,d \left( \omega^1_y+\sum
_{e\in\E^d
: y\in e}\omega_e+\delta_0 \right)(s). \]

We recall that for every $y\in\Z^d$ and $e\in\E^d$, $\omega_y^1$ and $\omega_e$ are the Poisson processes giving respectively the possible death times at $y$ and the possible times of births through $e$. Applying Lemmas 13 and 15 of~\cite{GMPCRE} to the process $(A_t)$ we obtain the following result.

\begin{lemma}
\label{boite}
If $N_L(x,t)\circ\theta_{s}=0$ and $s+ t\le u_i(x)\le s+L$, then either $v_i(x)=+\infty$ or $u_{i+1}(x)-u_i(x)\le\kappa t$.\\
Besides, for any $t,s>0$ and $x \in\Z^d$, the following inclusion holds:
\begin{align*}
&\{\tau^0=+\infty\}\cap\{\exists u\in x+B_{Mt+2}, \tau_u\circ\theta_s=+\infty,u\in A^0_s\} \nonumber\\
&\quad{}\cap\left\{N_{K(x)\kappa t}(x,t)\circ\theta_{s}=0\right\} \\
&\quad{} \cap\bigcap_{1\le i<K(x)} \{v_i(x)-u_i(x)< t\}\\
&\qquad\subset\{\tau^0=+\infty\} \cap\{\sigma(x)\le s+K(x)\kappa t \}.
\end{align*}
\end{lemma}

We now control the probability that a space-time box contains no bad growth point.
If we choose 
\begin{align}
 \label{kappa}
 \kappa=3M(1+C),
\end{align}
where $M$ and $C$ are the constants respectively given in Lemma~\ref{richardson} and Theorem~\ref{au-lineaire}, the next control follows.
\begin{lemma}
\label{controleboite}
There exist $A,B>0$ such that for all $L>0$, $x\in\Z^d$ and $t>0$ one has
\[\P\left(N_L(x,t)\ge1\right)\le A\exp(-Bt).\]
\end{lemma}

\begin{proof} Here again, the quantity $N_L(x,t)$ only considers the alive or dead nature of the points (and not their ages) so we work with $A_t=\supp\xi_t$. 
Let $x \in\Z^d$, $t>0$ and $y \in x+B_{M t+2}$. First, we control the probability of $E^y(x,t)$. There exist $A_1,B_1,A_2,B_2>0$ such that, for all $t>0$, one has
\begin{itemize}
 \item $\P\left(\omega_y([0, t/2])=0\right)= \exp(- t/2)$,
 \item $\P(H^y_{t}\not\subset y+B_{Mt})\le A_1\exp(-B_1t)$ by Lemma~\ref{richardson},
 \item $\P(t<\tau^x<\infty)\leq A_2\exp(-B_2t)$ by Theorem~\ref{petitsamas},
 \item If $\tau^y=+\infty$, there exists $z\in A^y_{2 t}$ with $\tau^z\circ\theta_{2t}=+\infty$.
Thus, the choice of $\kappa$ implies that
\begin{align*}
 & \left\{\tau^y=+\infty, \inf\{s\ge2t: x\in A^y_{s}\}> \kappa t\right\}\\
&\qquad \subset\{A^y_{2 t}\not\subset y+B_{2M t}\}
\cup\bigcup_{z\in y+B_{2M t}}\{ t^z(x)\circ\theta_{2t}>(\kappa-2M) t \}\\
&\qquad \subset \{A^y_{2t}\not\subset y+B_{2Mt}\}\cup\bigcup_{z\in y+B_{2M t}}\left\{t^z(x) \circ\theta_{2t}>C\|x-z\|+Mt-3C \right\}.
\end{align*}
Hence, with Lemma~\ref{richardson} and Theorem~\ref{au-lineaire},
\begin{align*}
 &\P(\tau^y=+\infty, \inf\{s\ge2 t: x\in A^y_{s}\}>\kappa t)\\
&\qquad\le A\exp(-2BM t)+(1+4Mt)^d A\exp\left(-B(Mt-3C)\right).
\end{align*}
\end{itemize}
So, we obtain the existence of $A_3,B_3>0$ such that for all $x\in\Z^d$, $t>0$ and for all $y\in x+B_{M t+2}$ we have
\[\P( E^y(x,t))\le A_3\exp(-B_3t).\]

For $y\in x+B_{Mt+2}$ we write $\beta_y=\omega_y^1+\sum_{e\in\E^d: y\in e}\omega_e$. $\beta_y$ is a Poisson point process with
intensity $1+2d\lambda_\infty$. We recall that $(T^\infty_n)_{n\ge1}$ is the increasing sequence of the times given by this process (with $T^\infty_0=0$). One has
\begin{align*}
 \P(N_L(x,t) \ge1) \le\E[N_L(x,t)] &= \sum_{y\in x+B_{Mt+2}}\E\left[\int_0^L \1_{E^y(x,t)}\circ\theta_s \,d(\beta_y+\delta_0)(s)\right]\\
 &= \sum_{y\in x+B_{Mt+2}}\E\left[\sum_{n=0}^{+\infty}\1_{\{T^\infty_n\le L\}}\1_{E^y(x,t)}\circ\theta_{T^\infty_n}\right]\\
  &= \sum_{y\in x+B_{Mt+2}}\sum_{n=0}^{+\infty}\E\left[\1_{\{T^\infty_n\le L\}}\right] \P(E^y(x,t))\\
   &\le \sum_{y\in x+B_{Mt+2}} \left(1+ \E[\beta_y([0,L])] \right) \P(E^y(x,t)) \\
   &\le (2Mt+5)^d \left(1+L(2d\lambda_\infty+1)\right)\P(E^y(x,t)).\qedhere
\end{align*}
\end{proof}

\subsection{Bound for the lack of subadditivity}

To bound $\sigma(x+y)-[\sigma(x)+\sigma(y)\circ\tilde{\theta}_x]$, we employ the procedure we have just mentioned around site $x+y$. To initiate the recursive process, one can benefit here from the existence of an infinite start at the specific point $(x+y, \sigma(x)+\sigma(y)\circ\tilde{\theta}_y)$.
We recall Theorem~\ref{presquesousadditif}:

\thmpresquesousadditif*

\begin{proof}Let $x,y \in\Z^d$ and $t>0$. We set $s=\sigma(x)+\sigma(y)\circ\tilde{\theta}_x$:
\begin{align*}
 \Pbarre\left( \sigma(x+y)>\sigma(x)+\sigma(y)\circ\tilde{\theta}_x+t\right)& \le  \Pbarre\left( K(x+y)> \frac{\sqrt t}{\kappa}\right)\\
&+\Pbarre\left(K(x+y) \le\frac{\sqrt t}{\kappa},\sigma(x+y)\ge s + t\right).
\end{align*}
Thanks to Lemma~\ref{K}, $\Pbarre\left( K(x+y)> \frac{\sqrt t}{\kappa}\right)\le \frac{1}{\rho}\exp\left(\frac{\sqrt{t}}{\kappa}\ln(1-\rho)\right)$.
For the second term we want to apply Lemma~\ref{boite}. Note that if $K(x+y) \le\frac{\sqrt t}{\kappa}$, then
$t \ge K(x+y) \kappa\sqrt t$ and $\tau^0=\infty$. Then 
\[\left\{N_{K(x+y) \kappa\sqrt t}\left(x+y,\sqrt t\right) \ge1\right\} \subset\left\{
N_{t}\left(x+y,\sqrt t\right) \ge1\right\}.\]
So using Lemma~\ref{boite} around $x+y$, on a scale $\sqrt t$, at initial time $s$ and with the source
point $u=x+y$,
\begin{align*}
\Pbarre & \left(\tau^0=+\infty, K(x+y) \le\frac{\sqrt t}{\kappa}, \sigma(x+y)\ge s + K(x+y)\kappa\sqrt t \right)\\
& \le \Pbarre\left(N_t\left(x+y, \sqrt t\right)\circ\theta_{s} \ge1\right)+\Pbarre\left(\exists i<K(x+y): v_i(x+y)-u_i(x+y) >\sqrt t\right)\\
& \le \Pbarre\left(N_t\left(x+y, \sqrt t\right)\circ\theta_{s} \ge1\right)+A\exp(-Bt)
\end{align*}
by Lemma~\ref{vimoinsui}.
Since $N_t(x+y,\sqrt{t})\circ{\theta_s}=N_t(0,\sqrt{t})\circ T_x\circ T_y\circ \theta_{\sigma(x)}\circ\theta_{\sigma(y)\circ\tilde{\theta}_x}$, we have
\[N_t\left(x+y,\sqrt t\right)\circ\theta_{s}=N_t\left(0,\sqrt{t}\right)\circ\tilde{\theta
}_y\circ\tilde{\theta}_x.\]

Thus
$ \Pbarre(N_t(x+y,\sqrt t)\circ\theta_s \ge1)= \Pbarre (N_t(0,\sqrt t)\ge1)$, which is controlled by Lemma~\ref{controleboite}.
\end{proof}

\begin{corollary}\label{momentsecart}
For $x,y \in\Z^d$, set $r(x,y)=(\sigma(x+y)-(\sigma(x)+\sigma(y)\circ\tilde{\theta}_x))^+$.
For any $p\ge1$, there exists $M_p>0$ such that for all $x,y\in\Z^d$, \[\E[r(x,y)^p]\le M_p.\]
\end{corollary}
\begin{proof}
We write $\overline{\E}[r(x,y)^p]=\int_0^{+\infty}pu^{p-1}\Pbarre(r(x,y)>u) \,du$ and use Theorem~\ref{presquesousadditif}.
\end{proof}
\subsection{Difference between $\sigma$ and $t$}
\begin{proposition}\label{difference}
There exist $A,B,\alpha>0$ such that for every $z>0$ and every $x\in\Z^d$, 
\[\Pbarre\left(\sigma(x)\ge t(x)+K(x)\left(\alpha\log(1+\|x\|) +z\right)\right)\le A\exp(-Bz).\]
\end{proposition}
\begin{proof} We proceed in the same spirit as previously. First, we introduce a new box of bad growth points. Then, we show that the probability of these points is small and we conclude by the fact that the quantity to control is dominated by this probability. For further details, see Proposition 17 of~\cite{GMPCRE}.
 \end{proof}
 
We now have all the elements to prove Theorem~\ref{differenceuniforme}:

\thmdifferenceuniforme*
\begin{proof} For every $p \ge1$, we will control $\overline{\E} \left(|\sigma(x)-t(x)|^p\right)$. According to Proposition~\ref{difference}, the random variable $V_x=\frac{\sigma(x)-t(x)}{K(x)}-\alpha\log (1+\| x\|)$ is stochastically dominated by a random variable $W$ with exponential moments. Using the Minkowski inequality, we have
\begin{align*}
(\Ebarre |\sigma(x)-t(x)|^p)^{\frac{1}{p}} & =(\Ebarre |K(x)(\alpha\log(1+\|x\|)+V_x)|^p)^{\frac{1}{p}}  \\
&\le \alpha\log(1+\|x\|)(\Ebarre K(x)^p)^{\frac{1}{p}}+(\Ebarre[ K(x)^p V_x^p])^{\frac{1}{p}}\\
& \le \alpha\log(1+\|x\|)(\Ebarre K(x)^p)^{\frac{1}{p}}+(\Ebarre K(x)^{2p} \Ebarre V_x^{2p})^{\frac{1}{2p}}\\
& \le \alpha\log(1+\|x\|)(\Ebarre G^p)^{\frac{1}{p}}+(\Ebarre G^{2p} \E W^{2p})^{\frac{1}{2p}},
\end{align*}
where $G$ is a geometric random variable which stochastically dominates $K(x)$ by Lemma~\ref{K}. So, there exists $C(p)>0$ such that for every $x \in\Z^d$, 
\[\overline{\E} \left(|\sigma(x)-t(x)|^p\right) \le C(p) \left(\log(1+\|x\|)\right)^{p}.\]
Let $p> d$. We have:
\[\sum_{x\in\Z^d} \Ebarre\frac{|\sigma(x)-t(x)|^p}{(1+\|x\|)^p}
\le C(p) \sum_{x\in\Z^d} \frac{(\log(1+\|x|))^{p}}{(1+\|x\|)^p}<+\infty.\]

So $(\frac{|\sigma(x)-t(x)|}{1+\|x\|})_{x \in\Z^d}$ is almost surely in $\ell^p(\Z^d)$ and thus goes to $0$.
\end{proof}

We deduce three fundamental properties of $\sigma$:

\begin{corollary}\label{au-lineairesigma}
There exist $A,B,C>0$ such that
\[\forall x \in\Z^d,\forall t>0,~ \Pbarre\left(\sigma(x)\ge C\|x\|+t\right) \le A\exp\left(B\sqrt{t}\right).\]
\end{corollary}

\begin{proof}
 Let $\alpha$ be given as in Proposition~\ref{difference}.
 \begin{multline*}
\Pbarre\Big(\sigma(x)> (C+1)\|x\|+t\Big)  \le \Pbarre\left(t(x)\ge C{\|x\|}+t/2 \right) +\Pbarre\left( K(x)>\frac1{2\alpha} \sqrt{\|x\|+t/2}\right)\\
+\Pbarre\left(\sigma(x)> \left(C+1\right)\|x\|+t,~t(x)< C{\|x\|}+t/2,~ K(x)\le \frac1{2\alpha} \sqrt{\|x\|+t/2}\right).
\end{multline*}
The first term is controlled with the estimate of Theorem~\ref{au-lineaire}, the second one with Lemma~\ref{K}
and for the last one note that if $K(x)\le\frac1{2\alpha} \sqrt{\|x\|+t/2}$ then
\begin{align*} 
K(x)[\alpha\log(1+\|x\|)+\alpha\sqrt{\|x\|+t/2}]&\le \frac1{2\alpha} \sqrt{\|x\|+t/2}[\alpha\sqrt{\|x\|}+\alpha\sqrt{\|x\|+t/2}]\\
&\le\|x\|+t/2,
\end{align*}
so the last term is smaller than  
\[ \Pbarre\left(\sigma(x)>t(x)+K(x)\left(\alpha\log(1+\|x\|) +\alpha\sqrt{\|x\|+t/2}\right)\right),\]
which is controlled by Proposition~\ref{difference}.
\end{proof}

\begin{corollary}
\label{momentsigma}
For every $p\ge1$, there exists a constant $C(p)>0$ such that
\[\forall x\in\Z^d,~\overline{\E}[\sigma(x)^p] \le C(p) (1+\|x\|)^{p}.\]
\end{corollary}

\begin{proof}
Using the Minkowski inequality, one has
\[
(\Ebarre[\sigma(x)^p])^{1/p}\le C\|x\|+\left(\Ebarre\left[\left(\left(\sigma(x)-C\|x\|\right)^+\right)^p\right]\right)^{1/p}.
\]
Moreover, by Corollary~\ref{au-lineairesigma}, \[\Ebarre\left[\left(\left(\sigma(x)-C\|x\|\right)^+\right)^p\right]=
\int_0^{+\infty} pu^{p-1}\Pbarre\left(\sigma(x)-C\| x\|>u\right) du<+\infty.\qedhere\]
\end{proof}

From the control of the tail of $\sigma$ together with the almost subadditivity, we deduce the uniform continuity of $\sigma$: 

\begin{corollary}\label{uniformesigma}
For every $\epsilon>0$, $\Pbarre$-a.s., there exists $R>0$ such that
\[\forall x,y \in\Z^d,~(\|x\|\ge R\text{ and }\|x-y\|\le\epsilon\|x\|)\Longrightarrow\left(|\sigma(x)-\sigma(y)|\le C\epsilon\|x\|\right).\]
\end{corollary}
\begin{proof} For $m \in\N$ and $\varepsilon>0$, we define the event
\[B_m(\epsilon)=\{\exists x,y \in\Z^d:\|x\|=m, \|x-y\|\le
\epsilon m \mbox{ and } |\sigma(x)-\sigma(y)|>C\epsilon m \}.\]
Noting that
\begin{align*}
B_m(\epsilon) \subset\mathop{\bigcup_{(1-\epsilon) m\le\|x\| \le(1+\epsilon)m}}_{\|x-y\|\le\epsilon m}\{\sigma(y-x)\circ\tilde{\theta}_{x}+r(x,y-x)>C\epsilon m\},
\end{align*}
we have
\begin{align*}
\Pbarre( B_m(\epsilon) )& \le \mathop{\sum_{(1-\epsilon) m\le\|x\| \le(1+\epsilon)m}}_{\|z\|\le \epsilon m} \Pbarre\left(\sigma(z)\circ\tilde{\theta}_{x}+r(x,z)>C\epsilon m\right)\\
& \le \mathop{\sum_{(1-\epsilon) m\le\|x\| \le(1+\epsilon)m}}_{\|z\|\le \epsilon m} \Pbarre\left(\sigma(z)>2C\epsilon m/3\right)+\Pbarre\left(r(x,y-x)>C\epsilon m/3\right).
\end{align*}
Using the controls of Corollary~\ref{au-lineairesigma} and Theorem~\ref{presquesousadditif}, we have
\begin{align*}
\Pbarre( B_m(\epsilon) )
& \le (1+2\epsilon m)^d\left(1+2(1+\epsilon) m\right)^dA\exp\left(-B\sqrt{C\epsilon m /3}\right)\\
& {} +A'\exp\left(-B'\sqrt{C'\epsilon m /3}\right).
\end{align*}
 We conclude the proof using the Borel-Cantelli lemma.
\end{proof}

\section{Asymptotic shape theorem}\label{STFA}
To obtain an asymptotic shape theorem we recall the almost subadditive theorem of Kesten and Hammersley (\cite{hamm74}):

\begin{theorem}
\label{kestenhamm}
Let $(\Omega,\mathcal{F},\P)$ be a probability space. On this space, we consider a collection $(X_n)_{n\ge1}$ of real random variables with finite second moments and collections $(Y_{n,p})_{n,p\ge1}$, $(X'_{n,p})_{n,p\ge 1}$ of real random variables such that
\[\forall n,p\ge1\qquad X_{n+p}\le X_{n}+X'_{n,p}+Y_{n,p}.\]
We assume that the following conditions are satisfied:
\begin{itemize} 
 \item For every $n,p$, $X'_{n,p}$ and $X_p$ have the same distribution.
 \item $\sup_{n,p} Corr(X_n,X'_{n,p})<1$.
 \item There exist a non decreasing sequence of reals $(A_p)_{p\ge1}$ with $\sum_{p=1}^{+\infty} \frac{A_p}{p^{2}}<+\infty$ and a constant $\beta>0$ such that:
 \begin{itemize}
\item for every $n$, $\E[(X_n+n\beta)^{-}]^2\leq A_n^2$,
\item for every $n,p$, $\E[ Y^2_{n,p}]\le A_{n+p}^2$.
\end{itemize}
\end{itemize}

Then $\frac{X_n}{n}$ converges in $L^2(\P)$ to a constant $\mu$. Moreover, for every $m\ge 1$, the subsequence $n_k=m2^k$ satisfies:
\[\forall \epsilon>0,~ \sum_{k=1}^{+\infty}\P(|X_{n_k}|>\epsilon)<+\infty.\]
If we assume in addition that there exists a constant $C$ such that for all $\epsilon>0$
\[\limsup_{n\rightarrow\infty}\sup_{p:|n-p|\leq\epsilon n}\frac{|X_n-X_p|}{n}\leq C\epsilon,\]
then $X_n/n$ converges a.s. to $\mu$.
\end{theorem}

Using this theorem, we can formulate a general result on some random variables $(\sigma(x))_{x\in\Z^d}$ to conclude an asymptotic shape theorem. 
\begin{theorem}
\label{metaTFA} Let $(\Omega,\mathcal{F},\P)$ be a probability space. Let $(\sigma(x))_{x\in\Z^d}$ be random variables with finite second moments and such that, for every $x\in\Z^d$, $\sigma(x)$ and $\sigma(-x)$ have the same distribution. Let $(s(y))_{y\in\Z^d}$ and $(r(x,y))_{x,y\in\Z^d}$ be collections of random variables such that:
\begin{enumerate}
 \item[Hyp 1:]\label{Hssadd} $\forall x,y\in\Z^d,~\sigma(x+y)\leq\sigma(x)+s(y)+r(x,y)$
with $s(y)$ having the same law as $\sigma(y)$, and being independent from $\sigma(x)$.
\item[Hyp 2:]\label{Hcontrolereste} $\forall x,y\in\Z^d,~\exists C_{x,y}$ and $\alpha_{x,y}<2$ such that $\E[r(nx,py)^2]\le C_{x,y}(n+p)^{\alpha_{x,y}}$.
\item[Hyp 3:]\label{Hau+lin} $\exists C>0$ such that $\forall x\in\Z^d,~ \P(\sigma(nx)>Cn\|x\|)\xrightarrow{n\rightarrow\infty} 0$
\item[Hyp 4:]\label{Hunif} $\exists K>0$ such that $\forall \epsilon>0, \P-p.s ~\exists M$ such that $(\|x\|\geq M\text{ and }\|x-y\|\leq K\|x\|)\Rightarrow \|\sigma(x)-\sigma(y)\|\leq \epsilon\|x\|$.
\item[Hyp 5:]\label{Hau-lin} $\exists c>0$ such that $\forall x\in\Z^d,~ \P(\sigma(nx)<cn\|x\|)\xrightarrow{n\rightarrow\infty} 0$
\end{enumerate}
Then there exists $\mu:\Z^d\to\R^+$ such that 
\[\lim_{\|x\|\rightarrow \infty} \frac{\sigma(x)-\mu(x)}{\|x\|}=0~ a.s.\]
Moreover, $\mu$ can be extended to a norm on $\R^d$ and we have the following asymptotic shape theorem: for all $\epsilon>0$, $\P$ almost surely, for all $t$ large enough,
\[(1-\epsilon)B_{\mu}\subset\frac{\tilde G_t}t\subset(1+\epsilon
)B_{\mu},\]
where $\tilde{G}_t=\{x\in\Z^d:\sigma(x)\le t\}+[0,1]^d$ and $B_\mu$ is the unit ball for $\mu$.
\end{theorem}

\begin{proof} \textbf{Step 1:} Convergence in one given direction: we fix $x\in\Z^d\setminus\{0\}$ and we choose for all $n,p\ge 1$, $X_n=\sigma(nx)$, $X'_{n,p}=s(px)$, $Y_{n,p}=r(nx,px)$, $A_n=C_{x,x}n^{\alpha_{x}}$ and the probability measure $\P=\Pbarre$. The hypotheses 1, 2 and 4 allow us to use Theorem~\ref{kestenhamm}. We obtain that for all $x\in\Z^d$, there exists $\mu(x)$ such that $\Pbarre$-a.s.: 
\[\lim_{n \to+\infty} \frac{\sigma(nx)}{n}= \lim_{n \to+\infty} \frac{\overline{\E}
\sigma(nx)}{n}=\mu(x).\]
This convergence also holds in $\L^2(\Pbarre)$.

\textbf{Step 2:} Then we extend $\mu$ to a semi-norm on $\R^d$: 
\begin{itemize}
 \item homogeneity: it holds that $\sigma(nx)$ and $\sigma(-nx)$ have the same distribution, so $\mu(x)=\mu(-x)$ and extracting subsequences we obtain that for all $k\in\Z$, $\mu(kx)=|k|\mu(x)$. 
 \item subadditivity: thanks to the hypotheses 1 and 2 and to the convergence in $\L^1(\P)$ we obtain that $\mu(x+y)\leq\mu(x)+\mu(y)$.
 \item extension to $\R^d$: first we extend $\mu$ to $\Q^d$ by homogeneity. Next, using the hypothesis~3, we extract a subsequence $(n_k)_k$ such that $\P(\sigma(n_kx)>Cn_k\|x\|)\leq \exp(-n_k)$ and we use the Borel-Cantelli lemma to obtain that $\mu(x)\leq C\|x\|$. This inequality allows us to extend $\mu$ to every $x\in\R^d$ by continuity.
\end{itemize}

\textbf{Step 3:} $\mu$ is in fact a norm: thanks to the hypothesis 5 and the Borel-Cantelli lemma again, we obtain that $\mu(x)\ge c\|x\|$ for every $x\in\Z^d$, and then for every $x\in\R^d$ by extension. So $\mu$ separates points. 

\textbf{Step 4:} Thanks to the hypothesis 4, we can show in a classical way the uniform convergence, that is
\[\lim_{\|x\|\rightarrow \infty} \frac{\sigma(x)-\mu(x)}{\|x\|} =0.\]
Let us briefly recall the main ideas. We assume by contradiction that there exists a sequence $(y_n)$ of points of $\Z^d$ converging to a direction $z'$ (which can be chosen rational) such that $\|\sigma(y_n)-\mu(y_n)\| \geq \epsilon \|y_n\|_1$. Then we project this sequence on $z'\Z$ and we use the almost convergence in the direction $z'$. We obtain a contradiction using the uniform continuity of $\sigma$ and $\mu$ (between the sequence $(y_n)$ and its projection).

\textbf{Step 5:} It is now easy to obtain a geometric asymptotic shape theorem: we suppose by contradiction that there exists an increasing sequence $(t_n)$, with $t_n\to+\infty$ and $\frac{G_{t_n}}{t_n}\not\subset(1+\epsilon)A_{\mu}$ and we obtain a subsequence $(x_n)$ such that $\sigma(x_n)\le t_n$ and $\mu(x_n)/t_n>1+\epsilon$, which contradicts the uniform convergence. We do the same thing with the reverse inclusion.
\end{proof}

Let us go back to the contact process with aging. We obtain the expected asymptotic shape theorem for the hitting time $t$:
\thmTFA*

\begin{proof} First, we use Theorem~\ref{metaTFA} to show that $\sigma$ satisfies an asymptotic shape theorem. We check the hypotheses of Theorem~\ref{metaTFA}. Thanks to Corollary~\ref{momentsigma}, $\sigma$ has finite second moment required (in fact $\sigma$ has finite $p$th moment for every $p\geq 1$ and the inequality satisfied by $\sigma$ in Corollary~\ref{momentsigma} allow us to conclude that the convergence of $\frac{\sigma(nx)}{n}$ also holds in $\L^p(\Pbarre)$). We take $s(y)=\sigma(y)\circ\tilde{\theta}_x$. The hypotheses 1 and 2 are satisfied thanks to Corollaries~\ref{invariancePbarre} and~\ref{momentsecart}. The hypothesis~3 that is the at least linear growth has been shown in Corollary~\ref{au-lineairesigma} and the hypothesis~5 of at most linear growth is immediately checked thanks to the at most linear growth for $t$ in Lemma~\ref{richardson}:
\[\Pbarre(\sigma(nx)< Cn\|x\|) \le \Pbarre(t(nx)< Cn\|x\|) \le \frac{A}{\rho}\exp(-BCn\|x\|).\]
Finally, the hypothesis 4 has been shown in Theorem~\ref{uniformesigma}.

Then we deduce the result for $t$ from the result for $\sigma$ thanks to Corollary~\ref{differenceuniforme}.
\end{proof}

\bibliographystyle{alpha}

\end{document}